\documentclass{amsart}
\usepackage{lmodern}
\usepackage[english]{babel}
\usepackage{amsmath}
\usepackage{amssymb}
\usepackage{mathrsfs}
\usepackage{color}
\usepackage{graphicx}
\usepackage{caption}
\usepackage{subcaption}
\usepackage{float}
\usepackage[all,cmtip]{xy}
\usepackage{bm}
\usepackage{algorithm}
\usepackage{enumitem}
\newlist{steps}{enumerate}{1}
\setlist[steps, 1]{label = Step \arabic*:}
\usepackage{pinlabel}
\newtheorem{theorem}{\bf Theorem}[section]
\newtheorem{lemma}[theorem]{\bf Lemma}
\newtheorem{definition}[theorem]{\bf Definition}
\newtheorem{corollary}[theorem]{\bf Corollary}
\newtheorem{proposition}[theorem]{\bf Proposition}
\newtheorem{remark}[theorem]{\bf Remark}
\newtheorem{obs}[theorem]{\bf Observation}
\newtheorem{example}[theorem]{\bf Example}

\newcommand{\rme}{\mathrm{e}}
\newcommand{\rmi}{\mathrm{i}}

\newcommand{\C}{\mathbb{C}}

\begin{document}

\title{Links of inner non-degenerate mixed functions, Part II}

\author{
Benjamin Bode}
\date{}

\address{Instituto de Ciencias Matemáticas (ICMAT), Consejo Superior de Investigaciones Científicas (CSIC), C/ Nicolás Cabrera, 13-15, Campus Cantoblanco, UAM, 28049 Madrid, Spain}
\email{benjamin.bode@icmat.es}








%

\maketitle
\begin{abstract}
Let $f:\mathbb{C}^2\to\mathbb{C}$ be an inner non-degenerate mixed polynomial with a nice Newton boundary with $N$ compact 1-faces. In the first part of this series of papers we showed that $f$ has a weakly isolated singularity and that its link can be constructed from a sequence of links $L_1, L_2,\ldots,L_N$, each of which is associated with a compact 1-face of the Newton boundary of $f$. In this paper, we offer a complete description of the links of singularities of inner non-degenerate mixed polynomials with nice Newton boundary. We show that a link arises as the link of such a singularity if and only if it is the result of the procedure from Part 1 for a sequence of links $L_1,L_2,\ldots,L_N$ that all satisfy certain symmetry conditions. We prove the same result for convenient, Newton non-degenerate mixed polynomials with nice Newton boundary. We also introduce the notion of P-fibered braids with $O$-multiplicities and coefficients, which allows us to describe the links of isolated singularities of inner non-degenerate semiholomorphic polynomials (as opposed to weakly isolated singularities).
\end{abstract}

 \vspace{0.2cm} 

   \noindent \textit{Keywords:} semiholomorphic polynomial; weakly isolated singularity; real algebraic link; P-fibered braid; inner Newton non-degenerate; mixed function; radially weighted homogeneous
   
     \vspace{0.2cm} 
     
      \noindent \textit{Mathematics Subject Classification 2020:} 
    57K10, 14P25, 14B05, 32S55.

\section{Introduction}
This article is a continuation of the work in \cite{AraujoBodeSanchez}. We will review some concepts and results in the introduction and at appropriate places in the later sections in order to keep this paper as self-contained as possible. However, for a complete account of the definitions and details on constructions of links of singularities we advise the reader to consult \cite{AraujoBodeSanchez}.

\subsection{Background}
There are several important differences between the study of singularities of complex plane curves and real polynomial mappings $f:\mathbb{R}^4\to\mathbb{R}^2$. The first of these differences pertains to the notion of an isolated singularity.  Let $\Sigma_f$ denote the set of critical points of $f$, that is, the set of points where the Jacobian matrix of $f$ does not have full rank. Let $V_f$ denote the variety of a polynomial map defined by the equation $f=0$ and suppose that the origin $O\in\mathbb{R}^4$ is a singularity of $f$. Then $O$ is a \textit{weakly isolated singularity} if there is a neighbourhood $U$ of $O$ such that $V_f\cap \Sigma_f\cap U=\{O\}$. It is called an \textit{isolated singularity} if there is a neighbourhood $U$ with $\Sigma_f\cap U=\{O\}$. It follows from the Łojasiewicz inequality \cite{loja} that the two notions are equivalent for complex plane curves, while in the real setting they are not \cite{milnor,ak}.

We are interested in topological properties of the variety near a (weakly) isolated singularity. Intersecting $V_f$ with a sufficiently small 3-sphere centered at the singularity results in a closed 1-dimensional submanifold, the link of the singularity. The so-called algebraic links, the isotopy classes of links that one obtains in this way from singularities of complex plane curves, are completely classified as certain iterated cables of torus links \cite{brauner, kahler, burau, burau2}. It is known that every link in $S^3$ is the link of a weakly isolated singularity of a real polynomial map \cite{ak}. However, it is not known which links are \textit{real algebraic}, that is, links of isolated singularities of real polynomial maps. Benedetti and Shiota conjecture that a link is real algebraic if and only if it is fibered \cite{benedetti}. One difficulty of this conjecture is that isolated singularities are exceedingly rare for real polynomial maps of these dimensions, making explicit constructions very challenging.

Because of this it is useful to borrow some techniques from the complex setting that allow us to deduce that a given polynomial has an isolated singularity. This is where we encounter another important difference between the real setting and the complex setting. Writing $f$ as a mixed function, i.e., a complex-valued function in complex variables $u$ and $v$ and their complex conjugates $\bar{u}$ and $\bar{v}$, allows the generalisation of many concepts from the complex setting. This leads for example to the definition of the Newton polygon of a mixed function and different notions of non-degeneracy that imply the existence of (weakly) isolated singularities. However, while some of these notions of non-degeneracy, such as inner non-degeneracy (see Definition~\ref{Newtoncond}) and partial non-degeneracy, are known to be equivalent to each other in the complex setting (in these dimensions), they are known to be different from each other for mixed functions \cite{eder2}. Furthermore, each of these types of non-degeneracies of mixed functions comes in two flavours, a weak one that implies a weakly isolated singularity and a strong one that implies an isolated singularity.

We showed in \cite{AraujoBodeSanchez} that inner non-degeneracy of a mixed function $f$ implies that $f$ has a weakly isolated singularity, while strong inner non-degeneracy implies an isolated singularity. In order to describe the link of the singularity we introduced another condition on the Newton boundary, which we call ``being nice''. We showed how the link of a singularity of an inner non-degenerate mixed function $f:\mathbb{C}^2\to\mathbb{C}$ with a nice Newton boundary with $N$ compact 1-faces can be constructed from a sequence of links $L_1,L_2,\ldots,L_N$, each of which is associated with a compact 1-face of the mixed Newton boundary of $f$. However, this is not a complete description of the links of such singularities. 

As in the complex setting there is an order on the set of compact 1-faces of the Newton boundary and we may associate to the $i$th 1-face a positive weight vector $P_i$ and a polynomial map $f_{P_i}$. This polynomial admits a certain radial symmetry (it is \textit{radially weighted homogeneous}). Together with the inner non-degeneracy and niceness condition this implies that its zeros are the cone over a link, which is the link $L_i$ associated to that compact 1-face. While \cite{AraujoBodeSanchez} explains how these links have to be combined to obtain the link of the singularity, it does not specify which link types can arise in this way as an $L_i$.

In this paper we study radially weighted homogeneous mixed functions and show that the links of singularities of such functions are exactly those that satisfy a certain even or odd symmetry. This allows us to completely characterise the links of weakly isolated singularities of inner non-degenerate mixed functions with nice Newton boundary. By generalising the concept of a P-fibered braid, which was introduced in \cite{bode:sat} to construct radially weighted homogeneous polynomials with isolated singularities, we can also describe those links that lead to isolated singularities. This addresses (but does not solve) the conjecture by Benedetti and Shiota about which links are real algebraic by providing new families of real algebraic links.

\subsection{Main results}
As in the first part of this series of papers we use $(u,v)$ as complex coordinates in $\mathbb{C}^2$ and on $S^3$, the unit 3-sphere in $\mathbb{C}^2$. If $v$ is non-zero, we may write $v=r\rme^{\rmi t}$ with $r>0$, $t\in[0,2\pi]$, and similarly $u=R\rme^{\rmi\varphi}$ for non-zero $u$.
\begin{definition}
Let $\tau_u:\mathbb{C}^2\to\mathbb{C}^2$ be the map $\tau_u(u,v)=(u,-v)$, $\tau_v:\mathbb{C}^2\to\mathbb{C}$, $\tau_v(u,v)=(-u,v)$, and $\tau_1:\mathbb{C}^2\to\mathbb{C}$, $\tau_1(u,v)=(-u,-v)$.

We say that a subset $L$ of $\mathbb{R}^4\cong\mathbb{C}^2$ is $u$-even if it is invariant (as a set) under $\tau_u$, i.e., $\tau_u(L)=L$. It is $v$-even if it is invariant under $\tau_v$. We say $L$ is even if it is $u$-even or $v$-even. We say that an even link $L$ is fixed-point-free if it does not intersect the set of fixed points of the corresponding transformation $\tau_u$ or $\tau_v$.\\
We say that $L$ is odd if $\tau_1(L)=L$.\\
\end{definition}
In this work, $L$ will usually be a link. In this case, the definition refers to the actual geometric curve that forms a link, not its isotopy class, which is its link type. If $M$ is a 3-dimensional submanifold of $\mathbb{R}^4\cong\mathbb{C}^2$ that is invariant under the symmetries above, we may say that a link type $L$ in $M$ has one of the symmetries above if there is a representative in its isotopy class in $M$ that has that symmetry. 

Note that being even and fixed-point-free in $M=S^3$ is equivalent to being 2-periodic with the set $(u,0)$ (or $(0,v)$) as the set of fixed points of the involution \cite{murasugi}. The even symmetry has been used in the construction of isolated singularities of radially weighted homogeneous polynomials \cite{bode:ralg}, while the odd symmetry has been used in \cite{looijenga}. We prove that for inner non-degenerate mixed polynomials with exactly one compact 1-face in their Newton boundary, including all radially weighted homogeneous polynomials, these two symmetries characterise all links of weakly isolated singularities. 

\begin{theorem}\label{thm:main1}
A link type $L\subset S^3$ arises as the link of a weakly isolated singularity of some inner non-degenerate mixed function $f$ with exactly one compact 1-face in its Newton boundary if and only if it is fixed-point-free even, or the union of a fixed-point-free even link and the corresponding set of fixed points, or an odd link type.
\end{theorem}

Among the family of mixed polynomials there is a family that can be seen as an interesting intermediate between the real (or mixed) setting and the complex case. We say that a mixed polynomial is \textit{semiholomorphic} if it is holomorphic with respect to one of the two complex variables. By convention we usually assume that a semiholomorphic polynomial is holomorphic with respect to the variable $u$. Links of singularities of semiholomorphic polynomials are naturally braided. Again, we find that certain symmetries characterise the braids that are obtained from polynomials with exactly one compact 1-face in its Newton boundary.

\begin{definition}
Consider the transformation $\tau_k:\mathbb{C}^2\to\mathbb{C}^2, \tau_k(u,v)\mapsto(u\rme^{\rmi \pi/k},-v)$, $k>0$. We say that a subset of $\mathbb{R}^4\cong\mathbb{C}^2$ is $k$-symmetric with $k\in \mathbb{N}$ if it is invariant (as a set) under $\tau_k$.

Let $B$ be closed a braid in $\mathbb{C}\times S^1$. We denote by $\tilde{s}$ the number of strands of $B$ that are not $\{0\}\times S^1$. We say that $B$ is divisor-symmetric if it is $2^K$-symmetric for some $K\in\mathbb{N}_0$ such that $2^K$ divides $\tilde{s}$.
\end{definition}


\begin{theorem}\label{thm:main1b}
A link type $L\subset S^3$ arises as the link of a weakly isolated singularity of some inner non-degenerate mixed polynomial $f$ with semiholomorphic principal part with exactly one compact 1-face in its Newton boundary if and only if it can be realised as the closure of a braid $B$ in $\mathbb{C}\times S^1$ that is $u$-even or divisor-symmetric or as the union of the closure of such a braid and its braid axis.
\end{theorem}

The symmetries in Theorem~\ref{thm:main1b} are special cases of those in Theorem~\ref{thm:main1}, since closures of $u$-even braids are $u$-even links in $S^3$, closures of 1-symmetric braids are odd links in $S^3$ and closures of $2^K$-symmetric braids with $K>0$ are $v$-even links. However, we should not expect that every $v$-even link is the closure of a divisor-symmetric braid or that every even or odd link type realises its symmetry via a braided representative.

In \cite{AraujoBodeSanchez} we showed that the link $L$ of a singularity of an inner non-degenerate mixed function $f$ with a so-called ``nice'' Newton boundary (cf. Definition~\ref{def:blabla}) can be obtained from links $L_i$, $i=1,2,\ldots,N$, which are associated with the compact 1-faces of the Newton boundary of $f$ if $N>1$. The result of this construction is denoted by $L([L_1,L_2,\ldots,L_{N-1}],L_N)$ and depends only on the isotopy class of $L_1$ in $\mathbb{C}\times S^1$, the isotopy class of $L_i$, $i=2,3,\ldots,N-1$, in $(\mathbb{C}\backslash\{0\})\times S^1$ and the isotopy class of $L_N$ in $S^1\times(\mathbb{C}\backslash\{0\})$. If the function $f$ has a semiholomorphic principal part, i.e., its face functions do not involve the variable $\overline{u}$ and are thus holomorphic in $u$, then there is a similar construction that yields the link of the singularity as the closure of a braid $B(B_1,B_2,\ldots,B_N)$ (or $B^\circ(B_1,B_2,\ldots,B_N)$ if $f$ is not $u$-convenient) obtained from braids $B_i$ that are again associated with the compact 1-faces of the Newton boundary. Again, the construction only depends on the braid isotopy class of each $B_i$, $i=2,3,\ldots,N$, in $(\mathbb{C}\backslash\{0\})\times S^1$ and $B_1$ in $\mathbb{C}\times S^1$. Details of the construction can be found in \cite{AraujoBodeSanchez}.
Combining this with Theorem~\ref{thm:main1} we obtain a complete characterisation of links of inner non-degenerate mixed functions with nice Newton boundary.

\begin{theorem}\label{thm:main2}
A link type $L$ arises as the link of a weakly isolated singularity of an inner non-degenerate mixed function with nice Newton boundary if and only if $L$ is as in Theorem~\ref{thm:main1} or it has a representative $L([L_1,L_2,\ldots,L_{N-1}],L_N)$ with $N>1$, where each of $L_1\subset\mathbb{C}\times S^1$, $L_i\subset(\mathbb{C}\backslash\{0\})\times S^1$, $i=2,3,\ldots,N-1$, $L_N\subset S^1\times\mathbb{C}$ is fixed-point-free even or odd.
\end{theorem}
Note that the different $L_i$, $i\in\{1,2,\ldots,N\}$, can satisfy different symmetry requirements. $L_1$ could be odd, for example, while $L_2$ is fixed-point-free even.

\begin{theorem}\label{thm:main2b}
A link type $L$ arises as the link of a weakly isolated singularity of an inner non-degenerate mixed function with semiholomorphic principal part if and only if it has a representative that is the closure of $B(B_1,B_2,\ldots,B_N)$ or $B^\circ(B_1,B_2,\ldots,B_N)$, where each braid $B_i$ is $u$-even or divisor-symmetric with $B_1\subset\mathbb{C}\times S^1$ and $B_i\subset(\mathbb{C}\backslash\{0\})\times S^1$ for all $i=2,3,\ldots,N$.
\end{theorem}

Since the construction of $L([L_1,L_2,\ldots,L_{N-1}],L_N)$ in Theorem~\ref{thm:main2} only depends on the isotopy classes of $L_1$ in $\mathbb{C}\times S^1$, the isotopy class of $L_i$, $i=1,2,\ldots,N-1$, in $(\mathbb{C}\backslash\{0\})\times S^1$ and the isotopy class of $L_N$ in $S^1\times(\mathbb{C}\backslash\{0\})$, the symmetry constraints on the $L_i$ can be understood as ``the isotopy class of $L_i$ in $(\mathbb{C}\backslash\{0\})\times S^1$ (or $\mathbb{C}\times S^1$ or $S^1\times(\mathbb{C}\backslash\{0\})$ in the case of $L_1$ and $L_N$, respectively) has an even or odd representative''. The corresponding statement for the braids $B_i$ in Theorem~\ref{thm:main2b} is analogous.

Although Oka's notion of Newton non-degeneracy in \cite{oka} is not equivalent to inner non-degeneracy (see Example~3.4 in \cite{AraujoBodeSanchez}), the two definitions are closely related. In particular, we find that the set of links of singularities of convenient, Newton non-degenerate mixed functions with nice Newton boundary is identical to the set of links associated with inner non-degenerate mixed functions with nice Newton boundary.


\begin{theorem}\label{cor:Oka}
A link type $L$ arises as the link of a weakly isolated singularity of a convenient, Newton non-degenerate mixed function with nice Newton boundary if and only if it is the link of a weakly isolated singularity of an inner non-degenerate mixed function with nice Newton boundary.
\end{theorem}



The results above pertain to weakly isolated singularities. Let now $f$ be a $u$-convenient, radially weighted homogeneous, semiholomorphic polynomial with weakly isolated singularity and link $L$. Then we know that the singularity is an isolated singularity if and only if $L$ is the closure of a so-called P-fibered braid \cite{bode:ralg, bode:sat}, which is equivalent to $f$ being strongly inner non-degenerate, see Proposition~\ref{prop:radPfib}. (Note that we do not claim that every closure of a P-fibered braid is the link of an isolated singularity. Being the link of a weakly isolated singularity of a radially weighted homogeneous polynomial comes with certain symmetry restrictions as seen in Theorem~\ref{thm:main1}. So the assumption that $L$ is the link of the singularity already puts restrictions on it.) We generalise this notion and define \textit{P-fibered braids with $O$-multiplicities and coefficients}, which allow us to describe the links corresponding to isolated singularities. This involves the definition of a \textit{compatible sequence} of P-fibered geometric braids with $O$-multiplicities and coefficients, which we give in Definition~\ref{def:compseq}.

\begin{theorem}\label{thm:main3}
A link type $L$ arises as the link of an isolated singularity of a $u$-convenient, strongly inner non-degenerate mixed polynomial $f:\C^2\to\C$ with exactly $N$ compact 1-faces and semiholomorphic principal part if and only if it has a representative that is the closure of $B(B_1,B_2,\ldots,B_N)$, where $B_1,B_2,\ldots,B_N$ forms a compatible sequence of P-fibered geometric braids with $O$-multiplicities and coefficients and every braid $B_i$ is $u$-even or divisor-symmetric.
\end{theorem}

We want to emphasize that even though this provides a description of the family of links of singularities, it does not really solve the question of which links arise in this way because it is not known which links are closures of P-fibered braids (with or without $O$-multiplicities and coefficients) and it is very difficult to check if a given sequence of braids is compatible.

However, Theorem~\ref{thm:main3} allows us to construct isolated singularities. This results in new families of links of isolated singularities, the yet unclassified real algebraic links. A special case of these families is the following result.
\begin{corollary}\label{cor:sub}
For any link $L$ there is a positive integer $M$ such that for all integers $m>M$ there is a satellite link $L'$ of the $(2,4m)$-torus link that is a real algebraic link and that contains $L$ as a sublink, i.e., a subset of the components of $L'$ is ambient isotopic to $L$.
\end{corollary}



%

\subsection{Outline of the paper}
The remainder of this paper is structured as follows. In Section~\ref{sec2} we study the symmetries of links of singularities of radially weighted homogeneous semiholomorphic polynomials, resulting in a proof of Theorem~\ref{thm:main1b}. Section~\ref{sec3} generalises these arguments to radially weighted homogeneous mixed functions, which proves Theorem~\ref{thm:main1}. A proof of Theorem~\ref{thm:main2} and Theorem~\ref{thm:main2b}, which generalise the results to general inner non-degenerate mixed polynomials and inner non-degenerate mixed polynomials with semiholomorphic principal part, respectively, can be found in Sections~\ref{sec4} and \ref{sec5}. In Section~\ref{sec5} we also prove Theorem~\ref{cor:Oka}. In Section~\ref{sec6} we introduce the notion of a P-fibered braid with $O$-multiplicities and coefficients. This leads to a proof of Theorem~\ref{thm:main3} and Corollary~\ref{cor:sub}. 

\section{Links of radially weighted homogeneous semiholomorphic polynomials}\label{sec2}

In this section we study the symmetries of the link types that can be associated with radially weighted homogeneous, semiholomorphic polynomials. For background on mixed functions, specifically on the Newton boundary of a mixed polynomial, weight vectors, semiholomorphic polynomials and radially weighted homogeneous polynomials, we point the reader to \cite{oka, AraujoBodeSanchez}. By \cite{AraujoBodeSanchez} the link type of a singularity of an inner non-degenerate function $f$ only depends on its Newton principal part $f_{\Gamma}$, the terms on the Newton boundary $\Gamma_f$. It follows that the link types that one obtains for radially weighted homogeneous, inner non-degenerate polynomials are the same as the link types coming from inner non-degenerate polynomials with exactly one compact 1-face in its Newton boundary.

We may write any mixed polynomial $f:\mathbb{C}^2\to\mathbb{C}$ as 
\begin{equation}
f(u,v)=\sum_{\mu_1,\mu_2,\nu_1,\nu_2\in\mathbb{N}_0}c_{\mu_1,\mu_2,\nu_1,\nu_2}u^{\mu_1}\bar{u}^{\mu_2}v^{\nu_1}\bar{v}^{\nu_2}.
\end{equation} 
The support of $f$ is denoted by $supp(f)$ and is defined as the set of integer lattice points $(\mu,\nu)\in\mathbb{N}_0^2$ such that there exist $\mu_1,\mu_2,\nu_1,\nu_2\in\mathbb{N}_0$ with $\mu=\mu_1+\mu_2$, $\nu=\nu_1+\nu_2$ and $c_{\mu_1,\mu_2,\nu_1,\nu_2}\neq 0$.

The Newton polygon of $f$ is then the convex hull of $supp(f)+(\mathbb{R}^+)^2$. The Newton boundary $\Gamma_f$ of $f$ is the union of the compact faces (0- or 1-dimensional) of the Newton polygon. The sum of the terms whose corresponding integer lattice points lie on $\Gamma_f$ is called the principal part of $f$ and denoted by $f_{\Gamma}$.

For every positive weight vector $P=(p_1,p_2)\in\mathbb{N}^2$ we obtain a linear map $\alpha_P$ from $supp(f)$ to $\mathbb{N}_0$ given by $\alpha_P(\mu,\nu)=p_1\mu+p_2\nu$. The sum of terms of $f$ on which $\alpha_P$ is minimal is called the face function $f_P$ associated with $P$. A mixed polynomial is called radially weighted homogeneous if there is a positive weight vector $P$ such that $f=f_{\Gamma}=f_P$, i.e., the support of $f$ is the Newton boundary, which consists of (at most) one compact 1-face. We denote the (minimal) value that $\alpha_P$ takes on $\Delta_P:=supp(f_P)$ by $d(P;f)$.

Weight vectors have a natural ordering. Writing $P=(p_{1},p_{2})$, $Q=(q_1,q_2)$, we define $P\succ Q$ if and only if $\tfrac{p_{1}}{p_{2}}>\tfrac{q_1}{q_2}$. The set of weight vectors $P$ for which $\Delta_{P}$ is a 1-face in the Newton boundary is denoted by $P_1,P_2,\ldots,P_N$, where $P_i\succ P_j$ if and only if $i<j$.

In the following we assume that a radially weighted homogeneous polynomial has exactly one compact 1-face in its Newton boundary. We will see later (Lemma~\ref{lem:radvertex}) that we do not miss any links by ignoring the polynomials whose Newton boundary is a single vertex.

Let $f:\mathbb{C}^2\to\mathbb{C}$ be a radially weighted homogeneous mixed function with weight vector $P=(p_1,p_2), \gcd(p_1,p_2)=1$ and exactly one compact 1-face in its Newton boundary. Let $(b,n)$ and $(s,a)$ with $a<n$ be the coordinates of the vertices $\Delta_1$ and $\Delta_2$ at the endpoints of this 1-face. We write $V_f$ for the variety of $f$, i.e., the set of points $(u,v)$ in $\mathbb{C}^2$ with $f(u,v)=0$.


\begin{lemma}\label{lem1}
Let $f:\mathbb{C}^2\to\mathbb{C}$ be a radially weighted homogeneous semiholomophic polynomial with weight vector $P=(p_1,p_2)$, $\gcd(p_1,p_2)=1$. Then $V_f$ is $u$-even if $p_1$ is even and $2^K$-symmetric if $p_1$ is odd, where $2^K$ is the largest power of 2 that divides $p_2$.
\end{lemma}
\begin{proof}
Since $f$ is radially weighted homogeneous, its support is contained in the set of intersection points between the integer lattice and the line through $(b,n)$ and $(s,a)$, which has slope $\tfrac{-p_1}{p_2}$. These intersection points are given by $(b,n)+j(p_2,-p_1)$, where $j=0,1,2,\ldots,\tfrac{s-b}{p_2}$.

Thus if $p_1$ is even, then the second coordinate entry $n-jp_1$, $j=0,1,2,\ldots,\tfrac{s-b}{p_2}$, has the same parity for all intersection points . Writing
\begin{equation}
f(u,v)=\sum_{j=0}^{(s-b)/p_2}\sum_{\nu_1+\nu_2=n-jp_1}c_{b+jp_1,0,\nu_1,\nu_2}u^{b+jp_1}v^{\nu_1}\bar{v}^{\nu_2}
\end{equation}
we obtain
\begin{equation}\label{eq:tauu}
f(\tau_u(u,v))=\sum_{j=0}^{(s-b)/p_1}\sum_{\nu_1+\nu_2=n-jp_1}(-1)^{n-jp_1}c_{b+jp_1,0,\nu_1,\nu_2}u^{b+jp_1}v^{\nu_1}\bar{v}^{\nu_2}.
\end{equation}
Since the parity of $n-jp_1$ is independent of $j$, we have that $f(\tau_u(u,v))$ is $(-1)^nf(u,v)$ and so $V_f$ is $u$-even.

If $p_1$ is odd, then the parity of $n-jp_1$ alternates with $j$. By definition we may write $p_2=2^Km$ for some odd natural number $m$. We may multiply $f$ by any power of $u$ without changing $p_1$, $p_2$, $n$, $a$ or any of the symmetries of $V_f$. Multiplying $f$ by $u^q$, $q\in\mathbb{N}$, however changes both $b$ and $s$ by adding $q$. We may thus choose $q$ such that $b$ is divisible by $2^{K+1}$. Then $b+jp_2$ is divisible by $2^K$ for all integers $j$, and the parity of $(b+jp_2)/2^K$ is the parity of $jm$, or simply that of $j$, since $m$ is odd. Thus

\begin{align}
f(\tau_{2^K}(u,v))&=\sum_{j=0}^{(s-b)/p_1}\sum_{\nu_1+\nu_2=n-jp_1}c_{b+jp_1,0,\nu_1,\nu_2}(u\rme^{2\rmi\pi/(2^K)})^{b+jp_1}(-v)^{\nu_1}(-\bar{v})^{\nu_2}\nonumber\\
&=\sum_{j=0}^{(s-b)/p_1}\sum_{\nu_1+\nu_2=n-jp_1}c_{b+jp_1,0,\nu_1,\nu_2}(-1)^{(b+jp_1)/2^K}u^{b+jp_1}(-1)^{n-jp_1}v^{\nu_1}\bar{v}^{\nu_2}\nonumber\\
&=\sum_{j=0}^{(s-b)/p_1}\sum_{\nu_1+\nu_2=n-jp_1}(-1)^{2j+n}c_{b+jp_1,0,\nu_1,\nu_2}u^{b+jp_1}v^{\nu_1}\bar{v}^{\nu_2}=(-1)^nf(u,v).
\end{align}
Therefore, $V_f$ is invariant under $\tau_{2^K}$.
\end{proof}


Note that this implies that if $p_1$ and $p_2$ are both odd, then the zeros of $g$ are 1-symmetric, that is, they are odd. If $p_1$ is odd and $p_2$ is even, then the zeros are invariant under $\tau_{2^K}$ and thus also invariant under $\tau_{2^K}^{2^K}$, which is the composition of $2^K$ copies of $\tau_{2^K}$. But $\tau_{2^K}^{2^K}$ sends $(u,\rme^{\rmi t})$ to $(u\rme^{\rmi \pi},\rme^{\rmi (t+2^{K}\pi)})=(-u,\rme^{\rmi t})$. It follows that the image of the zeros in $S^3$ is even in this case, too.

The focus of this paper lies on mixed polynomials that are inner non-degenerate.
\begin{definition}[{\cite[Definition 3.1]{AraujoBodeSanchez}}]\label{Newtoncond}
We say that $f$ is {\bf inner Newton non-degenerate} if both of the following conditions hold:
\begin{itemize}
    \item[(i)] the face functions $f_{P_1}$ and $f_{P_N}$ have no critical points in $V_{f_{P_1}} \cap (\C^2\setminus \{v=0\})$ and $V_{f_{P_N}}\cap (\C^2\setminus \{u=0\})$, respectively.
\item[(ii)] for each 1-face and non-extreme vertex $\Delta$, the face function $f_\Delta$ has no critical points in $V_{f_\Delta} \cap (\C^*)^{2}$.
\end{itemize}
\end{definition}

Usually we drop ``Newton'' from the term and simply say that $f$ is \textit{inner non-degenerate}.

This definition may be extended to mixed polynomials $f$ without any compact 1-faces in their Newton boundary. Then the Newton boundary of $f$ consists of a single vertex $\Delta$. In order for such a $f$ to be inner Newton non-degenerate, we require that $f_{\Delta}$ has no critical points in $\mathbb{C}^2\backslash\{(0,0)\}$ that are also zeros of $f_{\Delta}$. Note however that in this case the origin in $\mathbb{R}^4$ is not necessarily a singular point of $f$ as in the example of $f(u,v)=u$. If it is a singular point, it is a weakly isolated singularity by the same arguments as in \cite{AraujoBodeSanchez}.

The definition of the property of being radially weighted homogeneous includes the possibility of such mixed polynomials whose support is a single lattice point. As the following lemma shows, the links of these polynomials can also be obtained from polynomials with one compact 1-face in their Newton boundary.

\begin{lemma}\label{lem:radvertex}
Let $f:\mathbb{C}^2\to\mathbb{C}$ be a mixed polynomial that is inner non-degenerate and whose Newton boundary consists of a single lattice point. If $f$ has a weakly isolated singularity, then the link of the singularity of $f$ is a Hopf link.
\end{lemma}
\begin{proof}
By \cite{AraujoBodeSanchez} we may assume that the support of $f$ is a single lattice point that is $(0,1)$, $(1,0)$ or $(1,1)$. If it is $(0,1)$ or $(1,0)$ a direct calculation shows that the origin is a regular point of $V_f$.

Thus $supp(f)=\{(1,1)\}$ and $f(u,v)=auv+bu\bar{v}+c\bar{u}v+d\bar{u}\bar{v}$ for some $a,b,c,d\in\mathbb{C}$. We have the obvious solutions of $u=0$ and $v=0$, which form a Hopf link. Writing $u=R\rme^{\rmi \varphi}$ and $v=r\rme^{\rmi t}$, we can interpret $f$ as the product of $Rr$ and a polynomial in $\rme^{\rmi\varphi}$, $\rme^{-\rmi\varphi}$, $\rme^{\rmi t}$ and $\rme^{-\rmi t}$. It is a somewhat lengthy, but direct calculation to show that the second factor has no zeros if $f$ is inner non-degenerate.
\end{proof}



Note that the Hopf link satisfies the symmetry conditions from Theorem~\ref{thm:main1} and is therefore automatically included in the list of links in Theorem~\ref{thm:main2}. We may thus ignore the polynomials whose Newton boundary is a single vertex. From now on the term ``radially weighted homogeneous polynomial'' will be reserved for polynomials with exactly one compact 1-face.

In \cite{bode:ralg} we used certain functions $g:\mathbb{C}\times S^1\to\mathbb{C}$ that are polynomials in a complex variable $u$ as well as in $\rme^{\rmi t}$ and $\rme^{-\rmi t}$ to construct radially weighted homogeneous semiholomorphic polynomials. In \cite{AraujoBodeSanchez} we reversed this construction. We may associate to any radially weighted homogeneous mixed polynomial $f$ a polynomial $g:\mathbb{C}\times S^1\to\mathbb{C}$ in variables $u$, $\bar{u}$, $\rme^{\rmi t}$ and $\rme^{-\rmi t}$ by setting
\begin{equation}\label{eq:g}
g(u,\rme^{\rmi t})=r^{-d(P;f)/p_2}f(r^ku,r\rme^{\rmi t}),
\end{equation}
where $P=(p_1,p_2)$ is the weight vector corresponding to the 1-face of the Newton boundary of $f$ and $k=\tfrac{p_1}{p_2}$. Note that $g$ is constructed in such a way that it does not depend on $r$ anymore. Since $f$ is radially weighted homogeneous, the factors of $r$ on the right hand side cancel each other. Using the notation from Lemma~\ref{lem1}, the exponent $-d(P;f)/p_2$ may also be written as $ks+a$.

If $f$ is semiholomorphic and inner non-degenerate, the corresponding $g$-polynomial is holomorphic in $u$ and may be interpreted as a loop in the space of complex polynomials in one variable with degree $\deg_ug=\deg_uf$.

\begin{lemma}\label{lem:symfg}
Let $f:\mathbb{C}^2\to\mathbb{C}$ be a radially weighted homogeneous semiholomorphic polynomial and $g$ defined via Eq.~\eqref{eq:g}. Then the zeros of $f$ in $S^3$ are (fixed-point-free) $u$-even/(fixed-point-free) $v$-even/odd/$k$-symmetric if and only if the zeros of $g$ are.
\end{lemma}
\begin{proof}
This is immediate from Eq.~\eqref{eq:g}. 
\end{proof}

\begin{definition}[Oka \cite{oka}]
A face function $f_{\Delta}$ of a mixed polynomial $f$ for some compact face $\Delta$ (0- or 1-dimensional) of the Newton boundary $\Gamma_f$ is called {\bf Newton non-degenerate} if $V_{f_{\Delta}}\cap\Sigma_{f_{\Delta}}\cap(\mathbb{C}^*)^2=\emptyset$. It is called {\bf strongly Newton non-degenerate} if $\Sigma_{f_{\Delta}}\cap(\mathbb{C}^*)^2=\emptyset$. We say that $f$ is (strongly) Newton non-degenerate if $f_{\Delta}$ is (strongly) Newton non-degenerate for all compact faces $\Delta$ of $\Gamma_f$.
\end{definition}

Recall that a mixed polynomial is called $u$-convenient if its Newton boundary has a non-empty intersection with the horizontal axis in the integer lattice. It is called $v$-convenient if its Newton boundary has a non-empty intersection with the vertical axis. If it is both $u$-convenient and $v$-convenient, we call $f$ convenient.

\begin{lemma}\label{lem23}
Let $f:\mathbb{C}^2\to\mathbb{C}$ be a radially weighted homogeneous, inner non-degenerate, semiholomorphic polynomial  with weight vector $P=(p_1,p_2)$. If $f$ is $u$-convenient, then its link is the closure of a braid $B$ that is $u$-even if $p_1$ is even and that is $2^K$-symmetric if $p_1$ is odd, where $2^K$ is the largest power of 2 that divides $p_2$. If $f$ is not $u$-convenient, its link is the union of the closure of a braid $B$ as above and its braid axis.
\end{lemma}
\begin{proof}
 
The lemma is an immediate consequence of Lemma~\ref{lem1}, Lemma~\ref{lem:symfg} and the fact that the zeros of a $g$-polynomial that corresponds to a $u$-convenient, radially weighted homogeneous semiholomorphic polynomial with a weakly isolated singularity form a closed braid $B$ in $\mathbb{C}\times S^1$ (cf. Lemma~4.4 in \cite{AraujoBodeSanchez}) with the closure of $B$ equal to the link of the singularity if $f$ is $u$-convenient. If $f$ is not $u$-convenient, its link is the union of the braid axis and the closed braid.
\end{proof}

Since the case of $p_1$ being odd leads to the two possible cases of the zeros being $v$-even (if $p_2$ is even) or being odd (if $p_2$ is odd), this implies the following.

\begin{lemma}
Let $f:\mathbb{C}^2\to\mathbb{C}$ be a radially weighted homogeneous, inner non-degenerate, semiholomorphic polynomial with weight vector $P=(p_1,p_2)$. By \cite{AraujoBodeSanchez} it has a weakly isolated singularity. The link type of its singularity is even if $p_1$ is even or $p_2$ is even. Otherwise it is odd.
\end{lemma}

\begin{lemma}\label{lem:circles}
Let $B$ be a $2^K$-symmetric braid in $\mathbb{C}\times S^1$, with $K\geq 1$. Then $B$ is invariant (as a set) under the transformation $(u,\rme^{\rmi t})\mapsto(u\rme^{\rmi \pi/2^{K-1}},\rme^{\rmi t})$.  
\end{lemma}
\begin{proof}
This lemma follows directly from the definition of $k$-symmetric braids by applying the transformation $\tau_{2^K}$ twice.
\end{proof}

Consider now the $g$-polynomial associated with a semiholomorphic polynomial $f$ as in Lemma~\ref{lem23} such that the roots of $g$ in $\mathbb{C}\times S^1$ form a $2^K$-symmetric braid with $K\geq 1$ and $\tilde{s}$ strands that are not $0\times S^1$. Then Lemma~\ref{lem:circles} implies that for each value of $t\in[0,2\pi]$ the non-zero roots of $g(\cdot,\rme^{\rmi t})$ lie on concentric circles centered at $0\in\mathbb{C}$ and are equally distributed among each circle. In other words, $g(u,\rme^{\rmi t})$ is given by the product of some power of $u$ (to be precise, $u^0=1$ or $u$) and several factors of the form $u^{\tilde{s}/2^K}-a_j(t)$, where $j=1,2,\ldots,\tilde{s}/2^K$, is indexing the different circles, and $a_{j}(t):[0,2\pi]\to\mathbb{C}$ determines the radius of the circle and the phase shift relative to the $2^K$th roots of unity.

Note that for some values of $t$ the different circles can coincide. However, the roots on the circles have to remain distinct. The positions of roots and circles for different values of $t$ are illustrated in Figure~\ref{fig:circles}. Furthermore, the functions $a_j(t)$ are not necessarily $2\pi$-periodic, since $t$ going from 0 to $2\pi$ could induce a non-trivial permutation of the circles, so that $a_j(2\pi)=a_{j'}(0)$ for some $j'$, but not necessarily $j'=j$.


Also note that the symmetry from Lemma~\ref{lem:circles} implies that a strand of a $2^K$-symmetric braid with $K\geq 1$ intersects $0\times S^1\subset\mathbb{C}\times S^1$ if and only if it is equal to $0\times S^1$ and thus forms its own connected component.

\begin{figure}[H]
\centering
\hspace{-0.5cm}
\includegraphics[height=4cm]{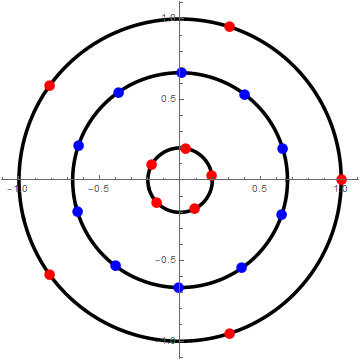}\qquad
\includegraphics[height=4cm]{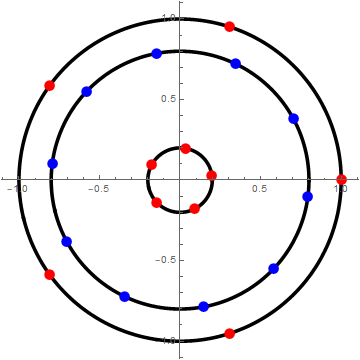}\\
\vspace{0.5cm}
\includegraphics[height=5cm]{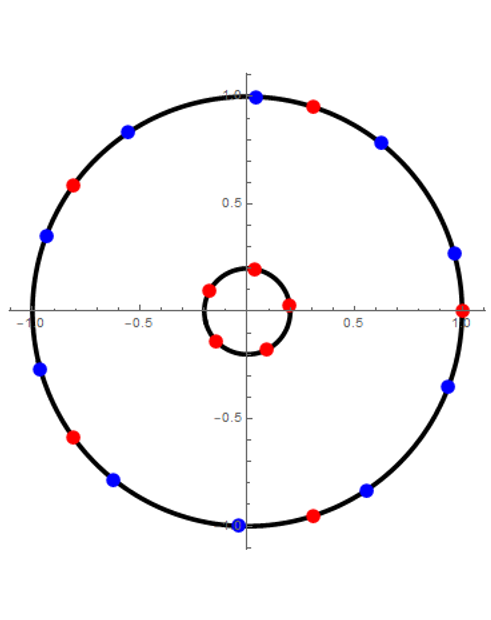}\quad
\includegraphics[height=5cm]{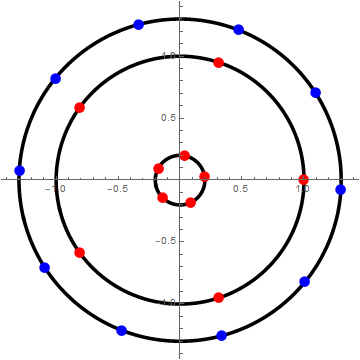}
\caption{Positions of roots of $g(\cdot,\rme^{\rmi t})$ for different values of $t\in[0,2\pi]$ if the resulting braid is $k$-symmetric. In this example $k=5$ and $\tilde{s}=20$. Roots are colored to distinguish different circles. \label{fig:circles}}
\end{figure}

\begin{lemma}
Let $f:\mathbb{C}^2\to\mathbb{C}$ be a radially weighted homogeneous, inner non-degenerate, semiholomorphic polynomial  with weight vector $P=(p_1,p_2)$. If $f$ is $u$-convenient, then its link is the closure of a braid $B$ that is $u$-even if $p_1$ is even and that is divisor-symmetric if $p_1$ is odd. If $f$ is not $u$-convenient, its link is the union of the closure of a braid $B$ as above and its braid axis.
\end{lemma}
\begin{proof}
Let $g$ be the $g$-polynomial associated with $f$ as in Eq.~\ref{eq:g}. We know from \cite{AraujoBodeSanchez} that the zeros of $f$ are the cone over the zeros of $g$, which form a closed braid $B$ in $\mathbb{C}\times S^1$.

Note that $p_2$ divides the degree of $f$ with respect to $u$ if $f$ is $v$-convenient and divides $\deg_uf-1$ if $f$ is not $v$-convenient. In both cases, these numbers are equal to $\tilde{s}$, the number of strands of $B$ that are not $0\times S^1$. Therefore, $p_2$ divides $\tilde{s}$. The result then follows from Lemma~\ref{lem23}, since $B$ is $u$-even or $2^K$-symmetric for a divisor $2^K$ of $p_2$, which is therefore also a divisor of $\tilde{s}$.
\end{proof}

\begin{lemma}\label{lem:uconv}
Let $f:\mathbb{C}^2\to\mathbb{C}$ be a $v$-convenient, inner non-degenerate mixed polynomial. Then $uf$ is also inner non-degenerate.
\end{lemma}
\begin{proof}
Let $P$ be a positive weight vector. Then the zeros of $(uf)_P$ in $\mathbb{C}^*\times\mathbb{C}$ are the same as the zeros of $f_P$ in $\mathbb{C}^*\times\mathbb{C}$. Thus the second condition in Definition~\ref{Newtoncond} is satisfied and the first condition is satisfied for $P=P_N$, the weight vector associated with the last 1-face of $uf$. Thus the only way that $uf$ could fail Definition~\ref{Newtoncond} would be if $(0,v_*)$ were a critical point of $(uf)_{P_1}$ in $V_{(uf)_{P_1}}$ for some $v_*\in\mathbb{C}^*$ and $P_1$ the weight vector associated with the first 1-face of $uf$.

We write $A(v)$ for $f(0,v)$, which is not constant 0, since $f$ is $v$-convenient. Then $(0,v_*)$ is a critical point of $(uf)_{P_1}$ if and only if $|\partial_v A(v_*)|=|\partial_{\bar{v}}A(v_*)|$ and $\partial_u(uf)(0,v_*)=A(v_*)=0$ \cite{AraujoBodeSanchez}. But that implies that $(0,v_*)$ is a critical point of $f_{P_1}$. We have $f_{P_1}(0,v_*)=A(v_*)=0$, which contradicts the inner non-degeneracy of $f$.
\end{proof}
By the same arguments as above $vf$ is inner non-degenerate if $f$ is $u$-convenient and inner non-degenerate.

\begin{lemma}\label{lem:symsemi}
Let $B$ be a braid in $\mathbb{C}\times S^1$ that is $u$-even. Then there is a radially weighted homogeneous, inner non-degenerate, semiholomorphic polynomial $f:\mathbb{C}^2\to\mathbb{C}$ with a weakly isolated singularity, whose link is the closure of $B$. There also exists such a polynomial whose link is the union of the closure of $B$ and its braid axis.
\end{lemma}
\begin{proof}
We may assume that $B$ does not contain a strand that is equal to $0\times S^1$, so that $\tilde{s}=s$, the number of strands of $B$. If $B$ does contain the zero-strand, we can simply multiply the polynomials that we construct in the following steps by $u$ as in Lemma~\ref{lem:uconv}.

Since $B$ is $u$-even, it can be written as $B=A^2$ for some braid $A$. In other words, if $g(u,\rme^{\rmi t})=\prod_{j=1}^s(u-u_j(\rme^{\rmi t}))$ denotes the loop of polynomials whose roots $u_j(\rme^{\rmi t})$, $t\in[0,2\pi]$, form the braid $A$, and whose coefficients are trigonometric polynomials as in \cite{bode:algo}, then we can obtain a corresponding loop of polynomials for $B$, say $\tilde{g}$, by multiplying $t$ by 2, so that each term $\rme^{\rmi t}$ becomes $\rme^{2\rmi t}$ and $\rme^{-\rmi t}$ becomes $\rme^{-2\rmi t}$. This means that all non-zero coefficients of $\tilde{g}$ are trigonometric polynomials that only involve even frequencies. As in \cite{bode:ralg} we obtain a radially weighted homogeneous, inner non-degenerate, semiholomorphic polynomial $f$ with weight vector $(k,1)$ and weakly isolated singularity, whose link is the closure of $B$, via
\begin{equation}\label{eq:fg}
f(u,v,\overline{v})=f(u,r\rme^{\rmi t},r\rme^{-\rmi t})=r^{ks}\tilde{g}\left(\frac{u}{r^k},\rme^{\rmi t}\right),
\end{equation}
where $s$ is the number of strands of $B$ and $k$ is some sufficiently large even integer.

Since the roots of $g(\cdot,\rme^{\rmi t})$ are simple for all $t\in[0,2\pi]$, the face function of the unique 1-face has no critical points in $V_f$ except for the origin. Thus $f$ is inner non-degenerate.

In order to obtain a polynomial whose link is the union of the closure of $B$ and its braid axis, we can multiply the polynomial for $B$ by $v$. By the analogue of Lemma~\ref{lem:uconv} the result is still inner non-degenerate.
\end{proof}

\begin{lemma}\label{lem:divsemi}
Let $B$ be a braid that is divisor-symmetric. Then there is a radially weighted homogeneous, inner non-degenerate, semiholomorphic polynomial $f:\mathbb{C}^2\to\mathbb{C}$ with a weakly isolated singularity, whose link is the closure of $B$. There also exists such a polynomial whose link is the union of the closure of $B$ and its braid axis.
\end{lemma}

\begin{proof}
By definition $B$ is $2^K$-symmetric for some divisor $2^K$ of $\tilde{s}$, the number of strands of $B$ that are not $0\times S^1$. As in the proof of Lemma~\ref{lem:symsemi} we may assume that $B$ does not contain a strand that is equal to $0\times S^1$, so that $\tilde{s}=s$, the number of strands of $B$. Otherwise, we can simply apply the following argument to $B\backslash(\{0\}\times S^1)$ and then multiply the resulting polynomial by $u$.

Suppose that $B$ is 1-symmetric, i.e., $B$ is odd and its number of strands $s$ is odd. Let $u_j:[0,2\pi]\to\mathbb{C}$, $j=1,2,\ldots,s$, denote the functions that parametrise the $s$ strands of $B$, so that $B$ is given as the set of parametric curves
\begin{equation}
\bigcup_{t\in[0,2\pi]}\bigcup_{j=1}^s(u_j(t),t)\subset\mathbb{C}\times[0,2\pi].
\end{equation}
Now consider the braid $B'$ that is obtained by multiplying each strand by $\rme^{\rmi t}$, i.e., 
\begin{equation}
\bigcup_{t\in[0,2\pi]}\bigcup_{j=1}^s(u_j(t)\rme^{\rmi t},t)\subset\mathbb{C}\times[0,2\pi].
\end{equation}
Since $B$ is odd, $B'$ is $u$-even and it follows that all coefficients of the loop of polynomials $g$ corresponding to $B'$ are even functions, i.e., $g(u,\rme^{\rmi t})=\sum_{j=0}^sc_j(\rme^{\rmi t})u^j$ with $c_j(\rme^{\rmi (t+\pi)})=c_j(\rme^{\rmi t})$ for all $j=1,2,\ldots,s$ and all $t\in[0,2\pi]$. Since the space of trigonometric polynomials with even frequencies is $C^1$-dense in the space of $\pi$-periodic functions, we can approximate each $c_j$ by such a trigonometric polynomial. Since distinct roots of a complex polynomial depend smoothly on its coefficients, the roots of this approximating polynomial still form the braid $B'$ (up to isotopy), albeit in a different parametrisation, say $u_j'(t)$ instead of $u_j(t)\rme^{\rmi t}$. Now define $U_j(t)=u_j'(t)\rme^{-\rmi t}$. These provide a new parametrisation of the strands of $B$ and the resulting loop of polynomials $\tilde{g}(u,\rme^{\rmi t}):=\prod_{j=1}^s(u-U_j(t))=\sum_{j=0}^sA_j(\rme^{\rmi t})u^j$ is seen to have coefficients $A_j(\rme^{\rmi t})$ that are trigonometric polynomials and that satisfy $A_j(\rme^{\rmi (t+\pi)})=(-1)^{s-j}A_j(\rme^{\rmi t})$. This implies that all frequencies of $A_j$ are odd when $j$ is even and even when $j$ is odd. Then we can define $f$ as in Eq.~\eqref{eq:fg} with the only difference that now $k$ is a sufficiently large odd integer.

Since the roots of $\tilde{g}(\cdot,\rme^{\rmi t})$ are distinct for all $t$, the polynomial has a weakly isolated singularity. The coefficient of $u^j$ in $f$ takes the form $(v\overline{v})^{k(s-j)}A_j\left(\tfrac{v}{\sqrt{v\overline{v}}}\right)$. It follows from the parity condition on the frequencies of $A_j$ that the resulting term is a semiholomorphic polynomial, i.e., all powers of $v\overline{v}$ are natural numbers. Thus $f$ is a semiholomorphic polynomial and by construction it is radially weighted homogeneous with weight vector $(k,1)$ and inner non-degenerate.

Lastly, suppose that $B$ is $2^K$-symmetric with $K>0$ and $2^K$ a divisor of $s$. By Lemma~\ref{lem:circles} its corresponding loop of polynomials $g$ is the product $\prod_{j=1}^{s/2^K}(u^{2^K}-a_j(t))$ for some functions $a_j(t):[0,2\pi]\to\mathbb{C}$, $j=1,2,\ldots,s/2^K$. Expanding the product yields $g(u,\rme^{\rmi t})=\sum_{j=0}^s u^jc_j(\rme^{\rmi t})$, where $c_j(\rme^{\rmi t})\neq 0$ implies that $j$ is divisible by $2^K$.

Since $B$ is $2^K$-symmetric, the roots of $g(\cdot,\rme^{\rmi (t+\pi)})$ are exactly the roots of $g(\cdot,\rme^{\rmi t})$ but rotated by $\rme^{\rmi \pi/2^K}$. Therefore $g(u\rme^{-\rmi \pi/2^K},\rme^{\rmi (t+\pi)})$ is equal to $g(u,\rme^{\rmi t})$ times some non-zero complex constant. Since $g(\cdot,\rme^{\rmi t})$ is monic for all $t$, this constant must be $-1$ if $s/2^K$ is odd and $+1$ if $s/2^K$ is even.

It then follows that $c_j(\rme^{\rmi(t+\pi)})=(-1)^{(s+j)/2^K}c_j(\rme^{\rmi t})$. We can thus approximate the function $c_j$ by a trigonometric polynomial with only even frequencies if $j/2^K$ has the same parity as $s/2^K$ and by a trigonometric polynomial with only odd frequencies if $j/2^K$ has a different parity from $s/2^K$. The zeros of the resulting loop of polynomials $\tilde{g}$ still form the braid $B$. As in the previous cases, we can now construct the corresponding radially weighted homogeneous semiholomorphic polynomial $f$ via Eq.~\eqref{eq:fg}, where $k$ is a sufficiently large natural number of parity $s/2^K$. The proof that $f$ is inner non-degenerate and thus has a weakly isolated singularity is completely analogous to the previous cases (the roots of $\tilde{g}(\cdot,t)$ are distinct for all $t$).

In all cases, we can obtain a polynomial whose link is the union of the closure of $B$ and its braid axis by multiplying the polynomial for $B$ by $v$. By the analogue of Lemma~\ref{lem:uconv} the result is still inner non-degenerate.
\end{proof}

Together with Lemma~\ref{lem1} this completely characterises the links of weakly isolated singularities of radially weighted homogeneous, inner non-degenerate semiholomorphic polynomials.

The arguments from \cite{AraujoBodeSanchez} imply that all results from this section remain true for inner non-degenerate mixed functions with semiholomorphic principal part and exactly one compact 1-face in their Newton boundary, since terms above the Newton boundary do not affect the link type of the singularity.

Before we turn our attention to mixed functions, we compare our findings for semiholomorphic polynomials with the well-established theory for holomorphic polynomials, which are of course a special case of our results. In particular, we want to emphasize the similarity (and differences) between the description of $2^K$-symmetric braids and torus braids. 

As above let $(u_j(t),t)\subset\mathbb{C}\times[0,2\pi]$, $j=1,2,\ldots,s$, $u_j:[0,2\pi]\to\mathbb{C}$, be a parametrisation of the strands of a $2^K$-symmetric braid. Then for each fixed value of $t$ the complex numbers $u_j(t)$, $j=1,2,\ldots,s$, lie on concentric circles, with $2^K$ of them placed symmetrically (equidistantly) along each circle. As $t$ varies, these points move on the circles, but also the radii of the circles are allowed to vary. In particular, the number of circles can vary with $t$ and circles can interchange their positions, resulting in non-trivial permutations as $t$ goes from $0$ to $2\pi$.

The holomorphic setting is a special case of this scenario, where the radii of the circles are not allowed to vary with $t$. Each circle for a fixed value of $t$ thus corresponds to a torus in $\mathbb{C}\times S^1$ and therefore each component of the braid closure must be a torus link. The tori are nested in precisely the way discussed in \cite{AraujoBodeSanchez}. Note that cable links do not arise in this way. They are the links of singularities of holomorphic polynomials whose Newton principal part is degenerate.

\section{Links of radially weighted homogeneous mixed functions}\label{sec3}

In this section we generalise the results from Section~\ref{sec2} to mixed functions. Many of the proofs will be conceptually very similar to the ones from the previous section and only deviate in some technical details.

\begin{lemma}\label{lem1mix}
Let $f:\mathbb{C}^2\to\mathbb{C}$ be a radially weighted homogeneous, mixed polynomial with weight vector $P=(p_1,p_2)$, $\gcd(p_1,p_2)=1$. Then $V_f$ is $u$-even if $p_1$ is even and $v$-even if $p_2$ is even. Otherwise $V_f$ is odd.
\end{lemma}
\begin{proof}

The proof follows the same reasoning as that of Lemma~\ref{lem1}. Using the same notation, the intersection points of the line through $(b,n)$ and $(s,a)$ with the integer lattice is again given by the set of points $(b,n)+j(p_2,-p_1)$ with $j=0,1,2,\ldots,\tfrac{s-b}{p_2}$.

Suppose that $p_1$ is even. Then all values of $\nu$ for which there is a $\mu\in\mathbb{N}_0$ with $(\mu,\nu)\in supp(f)$ have the same parity. Analogously, to Eq.~\eqref{eq:tauu} we have
\begin{equation}
f(\tau_u(u,v))=\sum_{j=0}^{(s-b)/p_1}\sum_{\nu_1+\nu_2=n-jp_1}\sum_{\mu_1+\mu_2=b+jp_2}(-1)^{n-jp_1}c_{\mu_1,\mu_2,\nu_1,\nu_2}u^{\mu_1}\bar{u}^{\mu_2}v^{\nu_1}\bar{v}^{\nu_2},
\end{equation}
where 
\begin{equation}
f(u,v)=\sum_{j=0}^{(s-b)/p_1}\sum_{\nu_1+\nu_2=n-jp_1}\sum_{\mu_1+\mu_2=b+jp_2}c_{\mu_1,\mu_2,\nu_1,\nu_2}u^{\mu_1}\bar{u}^{\mu_2}v^{\nu_1}\bar{v}^{\nu_2}.
\end{equation}
Since the parity of $n-jp_1$ is independent of $j$, we have $f(\tau_u(u,v))=(-1)^nf(u,v)$ and so $V_f$ is $u$-even.

By the same argument $V_f$ is $v$-even if $p_2$ is even.

Now suppose that both $p_1$ and $p_2$ are odd. Then the parity of $\mu=b+jp_2$ and of $\nu=n-jp_1$ alternates as $j$ increases from $0$ to $(s-b)/p_2$. We thus get
\begin{equation}
f(\tau_1(u,v))=\sum_{j=0}^{(s-b)/p_1}\sum_{\nu_1+\nu_2=n-jp_1}\sum_{\mu_1+\mu_2=b+jp_2}(-1)^{n+b+j(p_2-p_1)}c_{\mu_1,\mu_2,\nu_1,\nu_2}u^{\mu_1}\bar{u}^{\mu_2}v^{\nu_1}\bar{v}^{\nu_2}.
\end{equation}
Since both $p_1$ and $p_2$ are odd, the factor $(-1)^{n+b+j(p_2-p_1)}$ does not depend on $j$ and so $f(\tau_1(u,v))=(-1)^{n+b}f(u,v)$. In particular, $V_f$ is odd.
\end{proof}

Since a 3-sphere of any radius that is centered at the origin is invariant under the transformations $\tau_u$, $\tau_v$ and $\tau_1$, it follows that all links of singularities of radially weighted homogeneous mixed polynomials must be $u$-even, $v$-even or odd.

In the semiholomorphic case we obtained the stronger $2^K$-symmetry if $p_1$ is odd. In the general mixed case, this is not guaranteed because the complex conjugate $\overline{u}$ is affected by the transformation $(u,\rme^{\rmi t})\mapsto(u\rme^{\rmi \pi/2^K},\rme^{\rmi (t+\pi)})$ differently than $u$ (unless $K=0$). Therefore, the argument from Lemma~\ref{lem1} does not hold anymore.

\begin{lemma}\label{lem:convcases}
Let $f:\mathbb{C}^2\to\mathbb{C}$ be a radially weighted homogeneous, inner non-degenerate mixed polynomial with weight vector $P=(p_1,p_2)$, $\gcd(p_1,p_2)=1$ with $p_1$ even. Then the link of the singularity of $f$ is fixed-point-free $u$-even if $f$ is $u$-convenient. The link is the union of a fixed-point-free $u$-even link and the corresponding set of fixed points if $f$ is not $u$-convenient.
\end{lemma}
\begin{proof}
If $f$ is not $u$-convenient, then $f(u,0)=0$ for all $u\in\mathbb{C}$. Since it is inner non-degenerate, it has an isolated singularity and so its link consists of the union of the unknot $O$ formed by the set $v=0$ and remaining components, say $L'$. Since $O\cup L'$ must be an $u$-even link by Lemma~\ref{lem1mix} and $O$ is fixed pointwise by $\tau_u$, it follows that $L'$ must be $u$-even. Since it does not intersect $O$, it is fixed-point-free.

If $f$ is $u$-convenient, then there is a vertex $\Delta_2$ of the Newton boundary of $f$ on the horizontal axis. Since $p_1$ is even, it follows that the second coordinate $\nu$ of $(\mu,\nu)\in\mathbb{N}_0^2$ is even for all $(\mu,\nu)\in supp(f)$. This implies that $f_v(u,0)=f_{\bar{v}}(u,0)=0$ for all $u\in\mathbb{C}$. We already know that the link of the singularity of $f$ is $u$-even. We need to show that it does not intersect the set of fixed points $v=0$. Suppose that there is a $u_*\in\mathbb{C}^*$ such that $f(u_*,0)=f_{\Delta_2}(u_*,0)=0$. Note that $f_{\Delta_2}$ can be written as $f_{\Delta_2}(R\rme^{\rmi\varphi},v)=R^s\Phi(\rme^{\rmi \varphi})$ for some trigonometric polynomial $\Phi(\rme^{\rmi\varphi})$. Writing $u_*=R_*\rme^{\rmi \varphi_*}$, we have that $\Phi(\rme^{\rmi\varphi_*})=0$. But then
\begin{equation}
f_v(u_*,0)=f_{\bar{v}}(u_*,0)=\frac{\partial f}{\partial R}(u_*,0)=0,
\end{equation}
so that $(u_*,0)$ is a critical point of $f=f_P$ with $f(u_*,0)=f_P(u_*,0)=0$, contradicting $f$ being inner non-degenerate.
\end{proof}

Of course, the corresponding result holds if we swap $u$ and $v$ as well as $p_1$ and $p_2$. Thus Lemma~\ref{lem1mix} and Lemma~\ref{lem:convcases} show that every link of a weakly isolated singularity is of the form specified in Theorem~\ref{thm:main1}.

In order to complete the proof of Theorem~\ref{thm:main1} we need to construct radially weighted homogeneous, inner non-degenerate mixed polynomials for any link with the given symmetries. For this we need some results about approximations of manifolds by real algebraic varieties.
\begin{lemma}[Lemma 3 in \cite{calcut}]\label{calcut}
Let $V\subset\mathbb{R}^n$ be a codimension-1 real algebraic set with compact singular set. Let $M\subset\mathbb{R}^n$ be a proper smooth, codimension-1 submanifold such that for some $R>0$, $M\backslash D_R^n=V\backslash D_R^n$, where $D_R^n$ denotes the $n$-ball of radius $R$ centered at the origin in $\mathbb{R}^n$. Then there exists a non-singular real algebraic set $W\subset\mathbb{R}^n$ that is properly isotopic to $M$. In fact, we may suppose there is a smooth isotopy $h_t:\mathbb{R}^n\to\mathbb{R}^n$ and a radius $R'>0$ such that $h_0$ is the identity $h_1(M)=W$ and $|h_t(x)|=|x|$ for all $t\in[0,1]$ and all $x$ with $|x|\geq R'$.
\end{lemma}

The following lemma is a combination of typical approximation results in real algebraic geometry, including Lemma~\ref{calcut}. However, the way in which these classical results are combined to prove the lemma is far from trivial.

\begin{lemma}\label{lem:approx}
Let $L$ be a link in $\mathbb{C}\times S^1$. Then there is a complex-valued map $g:\mathbb{C}\times S^1\to\mathbb{C}$, $g(u,\rme^{\rmi t})=\sum_{j_1,j_2\geq 0}\sum_{\ell\in\mathbb{Z}}c_{j_1,j_2,\ell}u^{j_1}\overline{u}^{j_2}\rme^{\rmi \ell t}$ that is a polynomial in variables $u$, $\overline{u}$, $\rme^{\rmi t}$ and $\rme^{-\rmi t}$ such that the zeros of $g$ are an arbitrarily close $C^1$-approximation of $L$.
\end{lemma}
\begin{proof}
First note that there is a closed (compact without boundary), orientable surface $F$ in $\mathbb{R}^4$ that intersects $\mathbb{C}\times S^1$ transversely in $L$ and a number of split unknots. This surface can be obtained by considering $L$ as a link in $S^3\subset\mathbb{R}^4$ and pushing a Seifert surface of $L$ inside the 4-ball. A second Seifert surface can be pushed outside of the 4-ball, so that the intersection with $\mathbb{C}\times S^1$ contains the link $L$. The desired surface $F$ is then the union of these two Seifert surfaces joined along their common boundary $L$. Since $F$ is compact, there are more intersections with $\mathbb{C}\times S^1$, but after an isotopy of $F$ that fixes $L$ we can assume that these intersections occur at points whose $u$-coordinate has modulus bigger than some chosen constant $K$ and that the isotoped surface intersects $\{u:|u|>K\}\times D$ in a finite number of parallel disks, whose boundary circles are the claimed split unknots. Here $D$ denotes the unit disk.

Now identify $\mathbb{R}^4$ and the point at infinity with $S^4$. Then $F$ is a compact, framed submanifold of $S^4$ and therefore equal to the regular level set $\Phi^{-1}(y)$ for some smooth Morse map $\Phi:S^4\to S^2$ with regular value $y$. Now take two smooth, simple loops $\gamma_1$ and $\gamma_2$ on $S^2$ such that they are disjoint from $\Phi(\infty)$ and critical values of $\Phi$, and they intersect each other transversely in exactly two points, one of which is $y$. Such loops exist, since $y$ is a regular value and hence there is a neighbourhood of $y$ consisting of regular values.

Therefore, the preimage sets $\Phi^{-1}(\gamma_1)$ and $\Phi^{-1}(\gamma_2)$ are smooth 3-manifolds $M_1$ and $M_2$ in $S^4$ that intersect in $F$ (the level set of $y$) and some other components (the level set of the second intersection point of $\gamma_1$ and $\gamma_2$). Furthermore, $M_1$ and $M_2$ are compact, since $S^4$ is compact and thus $\Phi$ is proper.

Now we may deform $M_1$ and $M_2$ (and with them their intersections), while leaving $L\subset F\subset M_1\cap M_2$ fixed. First of all, we can deform $M_1$ and $M_2$ such that the additional components of their intersection do not meet $\mathbb{C}\times S^1$. Secondly, we can assume that the intersections of $M_1$ and $M_2$ with $\{u:|u|>K\}\times D$ are 3-balls, intersecting transversely in the disks $F\cap(\{u:|u|>K\}\times D)$. We now fix some sufficiently large numbers $K'$, $K''$ and push these 3-balls in direction of higher and higher $|u|$, so that the intersection of $M_i$, $i=1,2$, with $D^4_{K''}\backslash D^4_{K'}$ consists of a finite number of copies of $S^2\times [0,1]$.

We deform $M_1$ and $M_2$ inside $D^4_{K''}\backslash D^4_{K'}$, so that for some $R_1, R_2$ with $K'<R_1<R_2<K''$ every component $C_i$ of the 3-manifold $M_1$ is given by the equation $\text{Re}(u)=a_i$ for some chosen real number $a_i\neq a_j$ if $i\neq j$ and similarly, the components of $M_2$ in this interval of spheres are given by $\text{Im}(u)=b_i$ for some complex numbers $b_i$. The isotoped surface $F$ in this interval of spheres therefore has constant $u$-coordinate.

\begin{figure}
\centering
\labellist
\Large
\pinlabel $a)$ at 500 4500
\pinlabel $b)$ at 3750 4500
\pinlabel $c)$ at 500 2400
\pinlabel $d)$ at 3750 2400
\pinlabel ${\color{red}F}$ at 2360 3640
\pinlabel $\mathbb{C}\times S^1$ at 2430 3070
\pinlabel $S^3$ at 1850 3400
\pinlabel $D^4_{K'}$ at 5350 3400
\pinlabel $D^4_{K''}$ at 5600 3000
\endlabellist
\includegraphics[height=9cm]{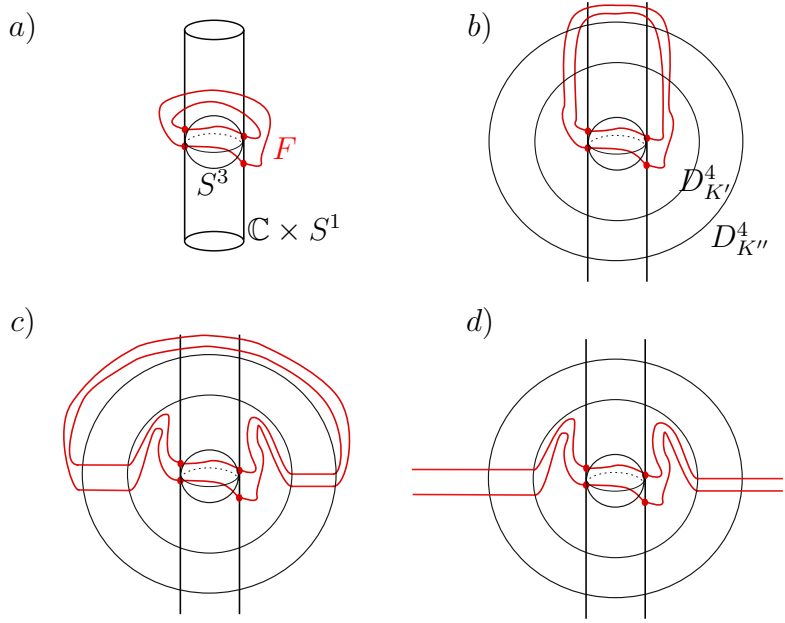}
\caption{Sketch of the deformation of the surface $F$. a) $F$ intersects $\mathbb{C}\times S^1$ in $L$ and a number of unknots. b) The unknots are pushed outside of some 4-ball $D^4_{K''}$. c) Between two 3-spheres $S^3_{K'}$, $S^3_{K''}$ of large radii the deformed $F$ can be taken to have constant $u$-coordinate. d) We obtain a surface that has constant $u$-coordinate outside some 4-ball $D^4_{K'}$ and intersects $\mathbb{C}\times S^1$ in $L$. \label{fig:calcut}}
\end{figure}

Remove $M_1\cap (\mathbb{R}^4\backslash D^4_{R_2})$ and $M_2\cap(\mathbb{R}^4\backslash D^4_{R_2})$ and continue the remainders of $M_1$ and $M_2$ in $M_1\cap (\mathbb{R}^4\backslash D^4_{R_2})$ as the solutions to the equations $\text{Re}(u)-a_i=0$ or $\text{Im}(u)-b_i=0$. Then $M_1$ and $M_2$ satisfy the conditions from Lemma~\ref{calcut} and can thus be isotoped to a real algebraic set.

The isotopy can be taken to be arbitrarily close to the identity on any compact subset of $\mathbb{R}^4$. In general, this is not true globally, specifically if the gradients $\nabla f_i$, $i=1,2$, of the polynomials $f_i$, $i=1,2$, that define $M_1$ and $M_2$ outside of some 4-ball go to the infinity. However, in our case, where the defining equations of $M_1$ and $M_2$ are linear, the gradients are constant outside some 4-ball. It follows from the proof of Lemma~3 in \cite{calcut} that the modulus of the vector field that is used to define the isotopy converges to zero outside some 4-ball. In particular, we can choose the isotopy arbitrarily close to the identity.

This implies that there is a polynomial $\tilde{g}=(\tilde{g}_1,\tilde{g}_2):\mathbb{R}^4\to\mathbb{R}^2$ such that $\tilde{g}_i^{-1}(0)$ approximates $M_i$. In particular, $\tilde{g}^{-1}(0)$ still intersects $\mathbb{C}\times S^1$ transversely in an arbitrarily close approximation of $L$. The desired polynomial $g$ is then simply the restriction of $\tilde{g}$ to $\mathbb{C}\times S^1$. 
\end{proof}

\begin{lemma}\label{lem:nondeg}
The polynomial $g(u,\rme^{\rmi t})=\sum_{j_1,j_2\geq 0}\sum_{\ell\in\mathbb{Z}}c_{j_1,j_2,\ell}u^{j_1}\overline{u}^{j_2}\rme^{\rmi \ell t}$ in Lemma~\ref{lem:approx} can be taken to satisfy the following property. Let $j$ be the maximum value of $j_1+j_2$ among all $(j_1,j_2)$ with $c_{j_1,j_2,\ell}\neq 0$. Then there is a unique pair $(j_1,j_2)$ with $j_1+j_2=j$ and $c_{j_1,j_2,\ell'}\neq 0$ for some $\ell'$. Furthermore, $\ell'$ is unique and equal to 0.
\end{lemma}
\begin{proof}
Start with the polynomial $\tilde{g}:\mathbb{R}^4\to\mathbb{R}^2$ constructed in Lemma~\ref{lem:approx}. We may write it as $\tilde{g}(u,\overline{u},v,\overline{v})=\sum_{j_1,j_2,k_1,k_2\geq 0}\tilde{c}_{j_1,j_2,\ell_1,\ell_2}u^{j_1}\overline{u}^{j_2}v^{\ell_1}\overline{v}^{\ell_2}$. Recall that the modulus of the $u$-coordinate of its zeros is bounded, because outside some 4-ball they are approximations of constants $u=a_i+\rmi b_i$ for some real numbers $a_i,b_i\in\mathbb{R}$. Now apply a change of variables that replaces the variable $u$ in the polynomial by $(v\overline{v})^ku$ (and accordingly replaces $\overline{u}$ by $(v\overline{v})^k\overline{u}$), where $k$ is a natural number. Note that this does not affect the zeros in $\mathbb{C}\times S^1$, where $v\overline{v}=1$. However, outside of the chosen 4-ball the zeros are now approximations to $\tfrac{a_i+\rmi b_i}{(v\overline{v})^k}$. The natural number $k$ is chosen such that $2k$ is bigger than the maximum value of $\ell_1+\ell_2$ with $\tilde{c}_{j_1,j_2,\ell_1,\ell_2}\neq 0$ for any pair $(j_1,j_2)$.

Let $j$ be the maximum value of $j_1+j_2$ with $\tilde{c}_{j_1,j_2,\ell_1,\ell_2}\neq 0$. Let $\ell$ denote the maximum value of $\ell_1+\ell_2$ with $\tilde{c}_{j_1,j_2,\ell_1,\ell_2}\neq 0$ and $j_1+j_2=j$. After the change of variables the largest value of $\ell_1+\ell_2$ with a non-zero coefficient is $2jk+\ell$. It is obtained by the terms 
\begin{equation}
\underset{j_1+j_2=j}{\sum_{j_1,j_2}}\underset{\ell_1+\ell_2=\ell}{\sum_{\ell_1,\ell_2}}\tilde{c}_{j_1,j_2,\ell_1,\ell_2}u^{j_1}\overline{u}^{j_2}v^{jk+\ell_1}\overline{v}^{jk+\ell_2}.
\end{equation}

Now perform a second change of variables that replaces the variable $u$ by $u+\varepsilon_1 (v\overline{v})^k$, where $\varepsilon_1>0$ is chosen sufficiently small. Since the zeros of $\tilde{g}$ are regular points intersecting $\mathbb{C}\times S^1$ transversely in $L$, the zeros of this new function still intersect $\mathbb{C}\times S^1$ in an arbitrarily close approximation of $L$ by choosing $\varepsilon_1$ small. The zeros outside the 4-ball $D^4_{R_1}$ now are approximations of $\tfrac{a_i+\rmi b_i-\varepsilon_1 (v\overline{v})^k}{(v\overline{v})^k}$, which converges to $-\varepsilon_1$ as $(v\overline{v})^k$ goes to infinity. In particular, the modulus of the $u$-coordinates of the zeros is still bounded.

The largest value of $\ell_1+\ell_2$ with non-zero coefficients is now $4jk+\ell$ and is obtained by the terms
\begin{equation}
\underset{j_1+j_2=j}{\sum_{j_1,j_2}}\underset{\ell_1+\ell_2=\ell}{\sum_{\ell_1,\ell_2}}\tilde{c}_{j_1,j_2,\ell_1,\ell_2}\varepsilon_1^j(v\overline{v})^{jk}v^{jk+\ell_1}\overline{v}^{jk+\ell_2}.
\end{equation}

We may now perform a last change of variable. We substitute $v$ by $v+\varepsilon_2 (u\overline{u})^J$, where $J$ is a natural number chosen such that $2J>j$ and $\varepsilon_2$ is now some complex number of small absolute value. Again, the transversality of $\tilde{g}^{-1}(0)$ and $\mathbb{C}\times S^1$ guarantees that their intersection is still $L$ after this change of variables.

The maximal value of $j_1+j_2$ with non-vanishing coefficients of the new polynomial is $2J(4jk+\ell)$ and is obtained by
\begin{equation}\label{eq:sumcoeff}
\varepsilon_1^j\underset{j_1+j_2=j}{\sum_{j_1,j_2}}\underset{\ell_1+\ell_2=\ell}{\sum_{\ell_1,\ell_2}}\tilde{c}_{j_1,j_2,\ell_1,\ell_2}\varepsilon_2^{2jk+\ell_1}\overline{\varepsilon_2}^{2jk+\ell_2}(u\overline{u})^{J(4jk+\ell)}.
\end{equation}
What is left to show is that the sum of coefficients $\tilde{c}_{j_1,j_2,\ell_1,\ell_2}\varepsilon_2^{2jk+\ell_1}\overline{\varepsilon_2}^{2jk+\ell_2}$ with $j_1+j_2=j$ and $\ell_1+\ell_2=\ell$ is not equal to 0.

We write $\varepsilon_2=|\varepsilon_2|\rme^{\rmi \chi}$. Then the sum above equals 0 if and only if
\begin{equation}
\underset{\ell_1+\ell_2=\ell}{\sum_{j_1+j_2=j}}\tilde{c}_{j_1,j_2,\ell_1,\ell_2}\rme^{\rmi \chi(\ell_1-\ell_2)}=0.
\end{equation}
The left hand side can be interpreted as the rational function
\begin{equation}
\frac{\underset{\ell_1+\ell_2=\ell}{\sum_{j_1+j_2=j}}\tilde{c}_{j_1,j_2,\ell_1,\ell_2}z^{\ell+\ell_1-\ell_2}}{z^{\ell}}
\end{equation}
restricted to the unit circle. The function is 0 if and only if its numerator, which is a polynomial of degree at most $2\ell$, is 0. Therefore there are only finitely many values of $\chi$ for which Eq.~\eqref{eq:sumcoeff} is 0 and we are free to choose a different value of $\chi$ as long as $|\varepsilon_2|$ is sufficiently small.
\end{proof}

If a radially weighted homogeneous mixed polynomial is semiholomorphic, we typically assume that it is holomorphic in the variable $u$. Then it makes sense to introduce the $g$-polynomial as in Eq.~\eqref{eq:g} or in \cite{AraujoBodeSanchez}, so that $g$ itself is also holomorphic in $u$. However, if $f:\mathbb{C}^2\to\mathbb{C}$ is a general mixed polynomial, there is no reason to choose the $u$-coordinate over $v$. Analogously to Eq.~\eqref{eq:g} we can associate to each radially weighted homogeneous mixed polynomial $f$ a polynomial $h:S^1\times\mathbb{C}\to\mathbb{C}$
\begin{equation}\label{eq:h}
h(\rme^{\rmi\varphi},v):=R^{-(k^{-1}n+b)}f(R\rme^{\rmi \varphi},R^{\tfrac{1}{k}}v,)=R^{-d(P;f)/p_1}f(u,R^{\tfrac{1}{k}}v).
\end{equation}
Just like in the definition of $g$ in Eq.~\eqref{eq:g} the polynomial $h$ is defined such that it does not depend on $R=|u|$.

It follows that $(\rme^{\rmi\varphi_*},v_*)=(\rme^{\rmi\varphi_*},r_*\rme^{\rmi t_*})\in S^1\times\mathbb{C}$ is a zero of $h$ if and only if 
\begin{equation}
f(R\rme^{\rmi \varphi_*},R\rme^{-\rmi\varphi_*},R^{\tfrac{1}{k}}v_*,R^{\tfrac{1}{k}}\overline{v_*})=0
\end{equation} 
for all $R\geq 0$, which is equivalent to $g(r_*^{-k}\rme^{\rmi \varphi_*},\rme^{\rmi t_*})=0$.


Naturally, Lemmas~\ref{lem:approx} and \ref{lem:nondeg} hold for links in $S^1\times\mathbb{C}$ and corresponding polynomial maps $h:\mathbb{C}\times S^1\to\mathbb{C}$ if we simply exchange the role of $u$ and $v$. As in Lemma~\ref{lem:symfg} the link of a singularity of a radially weighted homogeneous polynomial has a certain symmetry if and only if the zeros of the corresponding $h$-polynomial has that symmetry.

\begin{lemma}\label{lem:oddHopf}
Let $L$ be an odd link type in $S^3$. Then $L$ has an odd representative that does not intersect the sets $\{(u,0):u\in \mathbb{C}\}$ and $\{(0,v):v\in\mathbb{C}\}$.
\end{lemma}
\begin{proof}
Consider the Hopf fibration $H:S^3\to \mathbb{CP}=S^2$, $H(u,v)=u/v$ and the image $H(L)$ of $L$ under $H$. Since the image is 1-dimensional, there is a pair of antipodal points $p$ and $-p$ on $S^2$ that are not in $H(L)$. We can now apply a rotation to $S^2$ that maps $p$ to the north pole and $-p$ to the south-pole, whose preimages are the sets $\{(u,0):u\in\mathbb{C}\}$ and $\{(0,v):v\in\mathbb{C}\}$, respectively. The rotation of $S^2$ lifts to a rotation of $S^3$ and since the north pole and south pole are not in the image of the rotated $L$ under $H$, the rotated link does not intersect $\{(u,0):u\in \mathbb{C}\}$ and $\{(0,v):v\in\mathbb{C}\}$. Since the rotation of an odd link in $S^3$ is still an odd link, this proves the lemma.
\end{proof}

\begin{lemma}\label{lem:nonzerosym}
Let $L$ be a $u$-even in $\mathbb{C}\times S^1$. Then $L$ is isotopic in $\mathbb{C}\times S^1$ to another $u$-even link in $(\mathbb{C}\backslash\{0\}\times S^1)$.
\end{lemma}
\begin{proof}
Let $L$ be parametrised as $(u_j(\tau),\rme^{\rmi t_j(\tau)})$, where $j=1,2,\ldots,\ell$, is indexing the components of $L$ and the variable $\tau$ is going from $0$ to $2\pi$.

Now consider the curves $\bigcup_{j=1}^\ell\bigcup_{\tau\in[0,2\pi]}u_j(\tau)$ in $\mathbb{C}$. Since $L$ is smooth, there is an angle $\chi\in[0,2\pi]$ and a positive number $\varepsilon_0>0$ such that $-\epsilon\rme^{\rmi\chi}\neq u_j(\tau)$ for all $j$, all $\tau$ and all $0<\varepsilon<\varepsilon_0$. Thus $(u_j(\tau)+\epsilon\rme^{\rmi\chi},\rme^{\rmi t_j(\tau)})$ parametrises a link isotopic to $L$ that does not intersect $0\times S^1$. Furthermore, it is still $u$-even.
\end{proof}
Naturally, the same statement holds for $v$-even links in $S^1\times(\mathbb{C}\backslash\{0\})$.

Every link $L$ in $\mathbb{C}\times S^1$ corresponds to a link in $S^3$ by embedding $\mathbb{C}\times S^1$ in $S^3$ as the untwisted tubular neighbourhood of the unknot in $S^3$ given by $u=0$. This can be made precise in terms of parametrisations as follows. If $(u_j(\tau),\rme^{\rmi t_j(\tau}))$, $j=1,2,\ldots,n$, $\tau\in[0,2\pi]$, is a parametrisation of an $n$-component link $L$ in $\mathbb{C}\times S^1$, then $(\varepsilon u_j(\tau),\sqrt{1-\varepsilon^2|u_j(\tau)|^2}\rme^{\rmi t_j(\tau}))$, $j=1,2,\ldots,n$, $\tau\in[0,2\pi]$, is a parametrisation of the corresponding link in $S^3$ if $\varepsilon$ is a sufficiently small, positive real number. Note that if $L$ is $u$-even, $v$-even or odd, then the corresponding link in $S^3$ possesses that same symmetry. Likewise, every link in $S^3\backslash\{v=0\}$ corresponds to a link in $\mathbb{C}\times S^1$ that shares its symmetry properties.

\begin{lemma}\label{lem:symmix}
Let $L$ be a $u$-even or odd link in $(\mathbb{C}\backslash\{0\})\times S^1$. Then there is a radially weighted homogeneous, inner non-degenerate, mixed function $f:\mathbb{C}^2\to\mathbb{C}$ such that the zeros of the corresponding $g$-polynomial are isotopic to $L$ in $(\mathbb{C}\backslash\{0\})\times S^1$.

In particular, if we denote the corresponding even or odd link in $S^3$ by $L$ as well, then $L$ is the link of the weakly isolated singularity of $f$.
\end{lemma}
\begin{proof}
The proof follows a similar argument as that of Lemma~\ref{lem:symsemi}. If $L$ is $u$-even, there is a quotient link $L'\in (\mathbb{C}\backslash\{0\})\times S^1$ such that the image of $L'$ under the map $(u,\rme^{\rmi t})\mapsto(u,\rme^{2\rmi t})$ is $L$. Lemma~\ref{lem:approx} implies that $L'$ can be approximated arbitrarily well by the zero set $\tilde{g}^{-1}(0)$ of some polynomial $\tilde{g}:\mathbb{C}\times S^1\to \mathbb{C}$ in $u$, $\overline{u}$, $\rme^{\rmi t}$ and $\rme^{-\rmi t}$. In particular, the zero set is ambient isotopic to $L'$ in $(\mathbb{C}\backslash\{0\})\times S^1$. Furthermore, 0 is a regular value of $\tilde{g}$. Now consider the map $g:\mathbb{C}\times S^1\to\mathbb{C}$, $g(u,\rme^{\rmi t})=\tilde{g}(u,\rme^{2\rmi t})$. By construction $g$ is a polynomial in $u$ and $\overline{u}$, whose coefficients are polynomials in $\rme^{2\rmi t}$ and $\rme^{-2\rmi t}$, and its zero set is an approximation of (and in particular ambient isotopic to) the given link $L$.

Let $s$ be the total degree of $g$ with respect to $(u,\overline{u})$, i.e., if 
\begin{equation}
g(u,\rme^{\rmi t})=\sum_{j_1,j_2}c_{j_1,j_2}(t)u^{j_1}\bar{u}^{j_2},
\end{equation} then $s$ is the maximum $j_1+j_2$ with $c_{j_1,j_2}(t)\neq 0$. Let $a$ be the maximal degree (with respect to $\rme^{\rmi t}$ and $\rme^{-\rmi t}$) of all $c_{j_1,j_2}(t)$ with $j_1+j_2=s$.

Since all frequencies of the coefficients of $g$ are even, we can define $f$ as
\begin{equation}\label{eq:gscale}
f(u,\overline{u},v,\overline{v})=f(u,\overline{u},r\rme^{\rmi t},r\rme^{-\rmi t})=r^{ks+a}g\left(\frac{u}{r^k},\frac{\overline{u}}{r^k},\rme^{\rmi t}\right),
\end{equation}
where $k$ is a sufficiently large even integer.

The usual argument about the parities of the frequencies and $k$ implies that $f$ is a mixed polynomial and by construction it is radially weighted homogeneous. 
Since 0 is a regular value of $\tilde{g}$, it is also a regular value of $g$ and therefore the origin is the only critical point of $f$ on $f^{-1}(0)$. It is thus a weakly isolated singularity of $f$ and its link is $L$. The argument above already proves that the only compact 1-face of the Newton boundary of $f$ is non-degenerate.

Note that $f(0,0,v,\overline{v})\neq 0$ for all $v\in\mathbb{C}^*$, since $L$ does not intersect $\{0\}\times S^1$. Therefore, $f$ is $v$-convenient and the face function $f_{\Delta_1}$ is non-degenerate, where $\Delta_1$ is the vertex given by the intersection of the Newton boundary of $f$ and the vertical coordinate axis.

If we take the $g$-polynomial as in Lemma~\ref{lem:nondeg}, then $a=0$, so $f$ is $u$-convenient. The face function $f_{\Delta_2}$, where $\Delta_2$ is the vertex at the intersection of the Newton boundary and the horizontal axis, is then precisely Eq.~\eqref{eq:sumcoeff}, whose only zero is at $u=0$. It follows that $f_{\Delta_2}$ is non-degenerate. Therefore, $f$ is convenient and Newton non-degenerate and hence by \cite{AraujoBodeSanchez} also inner non-degenerate. 


Suppose $L$ is odd and its components are parametrised by $(u_j(\tau),\rme^{\rmi t_j(\tau)})$, where $j$ is labelling the different components, $u_j:[0,2\pi]\to\mathbb{C}\backslash\{0\}$ and $t_j:[0,2\pi]\to \mathbb{R}/(2\pi\mathbb{R})$ are $2\pi$-periodic functions, and $\tau$ is a variable going from 0 to $2\pi$. Since $L$ is odd, there is for every $j$ and every $\tau\in[0,2\pi]$ a $j'$ and a $\tau'\in[0,2\pi]$ with $(u_{j'}(\tau'),\rme^{\rmi t_{j'}(\tau')})=(-u_j(\tau),-\rme^{\rmi t_j(\tau)})$.
Now consider the link $L'$ whose components are parametrised by $(u_j(\tau)\rme^{\rmi t_j(\tau)},\rme^{\rmi t_j(\tau)})$. It follows from the above that for every $j$ and every $\tau$ there is a $j'$ and a $\tau'$ with
\begin{equation}
(u_{j'}(\tau')\rme^{\rmi t_{j'}(\tau')},\rme^{\rmi t_{j'}(\tau')})=(-u_j(\tau)\rme^{\rmi t_{j'}(\tau')},-\rme^{\rmi t_j(\tau)}),
\end{equation}
but since this implies that $\rme^{\rmi t_{j'}(\tau')}=-\rme^{\rmi t_j(\tau)}$, we get
\begin{equation}
(u_{j'}(\tau')\rme^{\rmi t_{j'}(\tau')},\rme^{\rmi t_j(\tau')})=(u_j(\tau)\rme^{\rmi t_j(\tau)},-\rme^{\rmi t_j(\tau)}),
\end{equation}
i.e., $L'$ is $u$-even.

We can now, as in the previous case, construct a $g$-polynomial for $L'$ whose coefficients $c_{j_1,j_2}(t)$ are trigonometric polynomials in $\rme^{2\rmi t}$ and $\rme^{-2\rmi t}$ and therefore satisfy $c_{j_1,j_2}(t+\pi)=c_{j_1,j_2}(t)$ for all $j_1,j_2$ and all $t$. The zeros of $g$ form the link $L'$, since $g(u,\rme^{\rmi t})=0$ if and only if there is a $j$ and a $\tau$ such that $(u,\rme^{\rmi t})=(\tilde{u}_j(\tau)\rme^{\rmi \tilde{t}_j(\tau)},\rme^{\rmi \tilde{t}_j(\tau)})$ for some approximations $\tilde{u}_j(\tau)$ of $u_j(\tau)$ and $\tilde{t}_j(\tau)$ of $t_j(\tau)$.

Now consider the polynomial $(u,\rme^{\rmi t})\mapsto g(u\rme^{-\rmi t},\rme^{\rmi t})$. The zeros of this function form the link $L$, i.e., $(u,\rme^{\rmi t})$ is a root if and only if $(u,\rme^{\rmi t})=(\tilde{u}_j(\tau),\rme^{\rmi \tilde{t}_j(\tau)})$ for some $j$ and $\tau$.
By construction the coefficients only contain even frequencies if $j_1+j_2$ is even and only odd frequencies if $j_1+j_2$ is odd. It follows that $f$ as in Eq.~\eqref{eq:gscale} defined from this $g$-polynomial results in a radially weighted mixed polynomial if $k$ is a sufficiently large odd integer.

We know from Lemma~\ref{lem:nondeg} that the leading term of $f$ is of the form $c(u\overline{u})^{s/2}$ for some $c\in\mathbb{C}^*$ (i.e., $a=0$). In particular, $f$ is $u$-convenient and there are no zeros with $v=0$. Since $0$ is a regular value of the $g$-polynomial for $L$, the mixed polynomial $f$ then has a weakly isolated singularity, whose link is $L$. This also establishes the non-degeneracy of the face function associated with the only 1-face of the Newton boundary of $f$.

Since $L$ is in $(\mathbb{C}\backslash\{0\})\times S^1$, any arbitrarily close approximation also has non-zero $u$-coordinates and therefore there are no roots of $f$ with $u=0$. Therefore, the inner non-degeneracy of $f$ follows from the non-degeneracy of the face function of its only 1-face.
\end{proof}

\begin{proof}[Proof of Theorem~\ref{thm:main1}]
By Lemmas~\ref{lem1mix} and \ref{lem:convcases} all links of radially weighted homogeneous, inner non-degenerate mixed polynomials have the required symmetries. By \cite{AraujoBodeSanchez} the link type of the singularity an inner non-degenerate mixed polynomials only depends on its principal part (the niceness condition from Definition~\ref{def:blabla} is always satisfied for polynomials with exactly one compact 1-face in its Newton boundary). Therefore, all links of inner non-degenerate mixed polynomials with exactly one compact 1-face in its Newton boundary have the desired symmetries.

Let now $L$ be an odd link type in $S^3$. Then by Lemma~\ref{lem:oddHopf} it has an odd representative that is disjoint from $\{(u,0):u\in \mathbb{C}\}$ and $\{(0,v):v\in\mathbb{C}\}$. It is thus the image of an odd link in $(\mathbb{C}\backslash\{0\})\times S^1$ under the usual embedding. (In fact, it is also the image of an odd link in $S^1\times(\mathbb{C}\backslash\{0\})$ under its usual embedding). Then by Lemma~\ref{lem:symmix} it can be realised as the link of a weakly isolated singularity of a radially weighted homogeneous, inner non-degenerate mixed polynomial.

If $L$ is 2-periodic, then there is an ambient isotopy of $S^3$ that takes the set of fixed points of the inversion to the unknot $S^3\cap\{(u,0):u\in\mathbb{C}\}$. Thus $L$ is the image of a $u$-even link in $\mathbb{C}\times S^1$. By Lemma~\ref{lem:nonzerosym} it has a representative that does not intersect $\{0\}\times S^1$ in $\mathbb{C}\times S^1$. Thus by Lemma~\ref{lem:symmix} it is realised as the link of a weakly isolated singularity of a radially weighted homogeneous, inner non-degenerate mixed polynomial.

If $L$ is the union of a 2-periodic link and a corresponding set of fixed points, we can apply the same arguments as in the previous case and then multiply the resulting mixed polynomial by $u$. The result is still radially weighted homogeneous (with the same weight vector) and inner non-degenerate by Lemma~\ref{lem:uconv}.
\end{proof}




We now briefly turn our attention to the smaller family of convenient and Newton non-degenerate mixed polynomials.

\begin{proposition}\label{prop:okarad}
A link type $L$ in $S^3$ arises as the link of a weakly isolated singularity of a radially weighted homogeneous, convenient, Newton non-degenerate mixed function if and only if $L$ is 2-periodic or odd.
\end{proposition}
\begin{proof}
We proved in \cite{AraujoBodeSanchez} that convenient, Newton non-degenerate mixed polynomials are inner non-degenerate. Therefore, links of singularities of convenient, Newton non-degenerate are 2-periodic or odd by Lemmas~\ref{lem1mix} and \ref{lem:convcases}.

The construction outlined in the previous lemma results in a convenient mixed polynomial if $L$ is 2-periodic or odd, since in both cases the link of the singularity is disjoint from the sets defined by $u=0$ and $v=0$. This also shows that $f_{\Delta_1}(0,v)\neq 0$ for all $v\in\mathbb{C}^*$ and $f_{\Delta_2}(u,0)\neq 0$ for all $u\in\mathbb{C}^*$, where $\Delta_1$ and $\Delta_2$ are the vertices of the Newton boundary of the constructed mixed polynomial $f$. This implies that $f_{\Delta_1}$ and $f_{\Delta_2}$ are Newton non-degenerate. Likewise $f_P=f$, the face function corresponding to the unique 1-face of $f$, is Newton non-degenerate, since critical points of $f_P$ correspond to critical points of the corresponding $g$-polynomial. But $g$ was constructed such that 0 is a regular value. It follows that the intersection of $\Sigma_f$ and $V_f$ is empty and therefore $f$ is Newton non-degenerate.
\end{proof}




\begin{example}\label{ex:knot}
Consider the knot $8_{16}$. We claim that it is neither 2-periodic nor odd and therefore cannot be realised as the link of a weakly isolated singularity of a radially weighted homogeneous, inner non-degenerate mixed function or of a radially weighted homogeneous, convenient, Newton non-degenerate mixed polynomial.

Knots that are 2-periodic must satisfy Murasugi's condition \cite{murasugi}, i.e., their Alexander polynomials $\Delta(t)$ must satisfy
\begin{equation}
\Delta(t)\equiv f(t)^2(1+t+t^2+\ldots+t^{\lambda-1}) \quad \text{mod }2
\end{equation}
for some Alexander polynomial $f(t)$ and some odd natural number $\lambda$. Here ``$\equiv$'' signifies equality up to multiplication by integer powers of $t$. The Alexander polynomial of $8_{16}$ is (according to knotinfo \cite{knotinfo})
\begin{equation}
\Delta(t)=1-4t+8t^2-9t^3+8t^4-4t^5+t^6=1+t^3+t^6\quad \text{mod }2
\end{equation}
and is irreducible modulo 2 (using the Mathematica software). Therefore, $f(t)$ would have to be constant, but then $\Delta$ is not of the form $1+t+t^2+\ldots+t^{\lambda-1}$ for any $\lambda$. Thus $8_{16}$ is not 2-periodic.

Odd links are freely periodic \cite{hartley}, which places another condition on the Alexander polynomial, namely
\begin{equation}
\Delta(t^2)=f(t)f(-t)
\end{equation}
for some Alexander polynomial $f(t)$. Since $\Delta(t^2)$ is irreducible (again verified with Mathematica), where $\Delta(t)$ is the Alexander polynomial of $8_{16}$, the knot is not freely periodic and hence not odd.
\end{example}


\section{Links of inner non-degenerate semiholomorphic polynomials}\label{sec4}

In this section, we extend Theorem~\ref{thm:main1b}, which discussed the case of semiholomorphic polynomials with exactly one compact 1-face in its Newton boundary, to semiholomorphic polynomials with any number $N$ of compact 1-faces.

We proved in \cite{AraujoBodeSanchez} that the link of the singularity of a $u$-convenient, inner non-degenerate semiholomorphic polynomial $f$ is the closure of a braid $B(B_1,B_2,\ldots,B_N)$ that can be explicitly constructed from a sequence of braids $B_1\subset\mathbb{C}\times S^1, B_i\subset(\mathbb{C}\backslash\{0\}\times S^1)$, $i=2,3,\ldots,N$, each of which is associated to a compact 1-face of the Newton boundary of $f$. If $f$ is not $u$-convenient, the link is the union of such a closed braid and its braid axis and is denoted by $B^o(B_1,B_2,\ldots,B_N)$. Since the link type of the singularity only depends on the principal part of $f$, the same is true for any inner non-degenerate mixed polynomial with semiholomorphic principal part.

The different braids $B_i\subset\mathbb{C}\times S^1$, $i=2,\ldots,N$, are exactly the zeros of the $g$-polynomials in $(\mathbb{C}\backslash\{0\}\times S^1)$ corresponding to the face functions $f_{P_i}$, where $P_i$ are the weight vectors corresponding to the 1-faces of $f$ in the usual ordering \cite{AraujoBodeSanchez}. The braid $B_1$ is given by the zeros of the $g$-polynomial associated with $f_{P_1}$.

\begin{lemma}
Let $f:\mathbb{C}^2\to\mathbb{C}$ be an inner non-degenerate mixed polynomial with semiholomorphic principal part and with $N$ compact 1-faces in its Newton boundary. Then the link of its singularity is the closure of $B(B_1,B_2,\ldots,B_N)$ or $B^o(B_1,B_2,\ldots,B_N)$, depending on whether $f$ is $u$-convenient or not, for a sequence of braids $B_1\subset\mathbb{C}\times S^1, B_i\subset(\mathbb{C}\backslash\{0\}\times S^1)$, $i=2,3,\ldots,N$, each of which is $u$-even or divisor-symmetric.
\end{lemma}
\begin{proof}
We already know from \cite{AraujoBodeSanchez} that the link is of the form $B(B_1,B_2,\ldots,B_N)$ or $B^o(B_1,B_2,\ldots,B_N)$ for a sequence of braids $B_1\subset\mathbb{C}\times S^1, B_i\subset(\mathbb{C}\backslash\{0\}\times S^1)$, $i=2,3,\ldots,N$. We need to show that each $B_i$, $i=1,2,\ldots,N$, has the desired symmetry.

Since $f$ is inner non-degenerate and semiholomorphic, it follows that $f_{\Delta}$ has no zeros in $(\mathbb{C}^*)^2$ for all non-extreme vertices $\Delta$ of the Newton boundary of $f$ and the extreme vertex $\Delta=\Delta_{N+1}$, see \cite[Lemma 4.3]{AraujoBodeSanchez}.

The same arguments as in Lemma~\ref{lem23} show that each $B_i$, $i=1,2,\ldots,N$, is $u$-even or divisor-symmetric.
\end{proof}

\begin{lemma}\label{lem:nonzeroodd}
Let $L$ be an odd link in $\mathbb{C}\times S^1$ (or $S^1\times\mathbb{C}$). Then $L$ is isotopic in $\mathbb{C}\times S^1$ (or $S^1\times\mathbb{C}$) to an odd link in $(\mathbb{C}\backslash\{0\})\times S^1$ (or $S^1\times (\mathbb{C}\backslash\{0\})$).
\end{lemma}
\begin{proof}
Consider a parametrisation $\cup_{j=1}^{\ell}\cup_{\tau\in[0,2\pi]}(u_j(\tau),\rme^{\rmi t_j(\tau)})$ of $L$, where $\ell$ denotes the number of components of $L$. Since $L$ is odd, the link parametrised by $(u_j(\tau)\rme^{\rmi t_j(\tau)},\rme^{\rmi t_j(\tau)})$ is $u$-even. By the same arguments as in Lemma~\ref{lem:nonzerosym} there is a $\chi\in[0,2\pi]$ such that $(u_j(\tau)\rme^{\rmi t_j(\tau)}+\varepsilon\rme^{\rmi \chi},\rme^{\rmi t_j(\tau)})$ is $u$-even and does not intersect $\{0\}\times S^1$ for all sufficiently small $\varepsilon>0$. Then $(u_j(\tau)+\rme^{-\rmi t_j(\tau)}\varepsilon\rme^{\rmi \chi},\rme^{\rmi t_j(\tau)})$ is odd and in $(\mathbb{C}\backslash\{0\})\times S^1$. Since $\varepsilon>0$ can be chosen arbitrarily small, the parametric curves describe a link that is isotopic to $L$.
\end{proof}

\begin{lemma}\label{lem:verttimes}
Let $f:\mathbb{C}^2\to\mathbb{C}$ be a Newton non-degenerate mixed function and let $\tilde{f}:\mathbb{C}^2\to\mathbb{C}$ be a mixed function, such that $supp(\tilde{f})$ is a single point, and without any zeros in $(\mathbb{C}^*)^2$. Then $f\tilde{f}$ is Newton non-degenerate. 
\end{lemma}
\begin{proof}
Note that $(f\tilde{f})_P=f_P\tilde{f}_P$ for all weight vectors $P$. Since $f$ is Newton non-degenerate, there are no critical points of $f_P$ in $V_{f_P}\cap (\mathbb{C}^*)^2$. But since $\tilde{f}=\tilde{f}_P$ is non-vanishing on $(\mathbb{C}^*)^2$, we have $V_{f_P}\cap (\mathbb{C}^*)^2=V_{(f\tilde{f})_P}\cap(\mathbb{C}^*)^2$. For all $(u_*,v_*)\in V_{f_P}\cap (\mathbb{C}^*)^2$, we obtain $\partial_x (f\tilde{f})_P(u_*,v_*)=\tilde{f}_P(u_*,v_*)\partial_x f_P(u_*,v_*)$, where $x\in\{u,\bar{u},v,\bar{v}\}$. Then since, $(u_*,v_*)$ is a regular point of $f_P$, it is also a regular point of $(f\tilde{f})_P$ and so $f\tilde{f}$ is Newton non-degenerate.
\end{proof}

\begin{lemma}\label{lem:semiglue}
Let $B_1,B_2,\ldots,B_N$ be a sequence of braids, each of which is $u$-even or divisor-symmetric with $B_1\subset\mathbb{C}\times S^1$ and $B_i\subset(\mathbb{C}\backslash\{0\})\times S^1$. Then there is an inner non-degenerate mixed polynomial with semiholomorphic principal part and with $N$ compact 1-faces in its Newton boundary whose link of the singularity is the closure of $B(B_1,B_2,\ldots,B_N)$. Likewise, there exists such a polynomial for $B^o(B_1,B_2,\ldots,B_N)$.
\end{lemma}
\begin{proof}
We may assume that $B_1$ does not contain a strand that is $\{0\}\times S^1$. If it does, we construct a polynomial for the remaining strands and multiply the result by $u$. If $B_1$ is $u$-even and intersects $\{0\}\times S^1$, we may also pick a different $u$-even representative of $B_1$ that does not have such intersections by the same arguments as in Lemma~\ref{lem:nonzerosym}.

By Lemmas~\ref{lem:symsemi} and \ref{lem:divsemi} we have for every braid $B_i$ a radially weighted homogeneous, inner non-degenerate semiholomorphic polynomial $f_i$, such that the zeros of the corresponding $g$-polynomial form the closed braid $B_i$. In fact, by the construction in Section~\ref{sec2} we may take $f_i$ to be convenient and Newton non-degenerate. We prove by induction on $N$ that we can combine these different radially weighted homogeneous polynomials to obtain an inner non-degenerate polynomial with the desired link of the singularity.

The case of $N=1$ is given by Theorem~\ref{thm:main1b}. Now assume that for every sequence of length $N-1$ of braids with the required symmetries we obtain a link of the singularity of a convenient, Newton non-degenerate semiholomorphic polynomial $\tilde{f}$ with $N-1$ compact 1-faces in its Newton boundary. In particular, for the sequence $B_i\subset(\mathbb{C}\backslash\{0\})\times S^1$, $i=2,3,\ldots,N$, such an $\tilde{f}$ exists. Since $B_2$ does not intersect $0\times S^1$, the polynomial $\tilde{f}$ is convenient and $\tilde{f}_{\Delta}$ has no zeros in $\mathbb{C}\times \mathbb{C}^*$, where $\Delta$ is the vertex of the Newton boundary of $\tilde{f}$ on the vertical axis.

Similarly, there is a convenient, Newton non-degenerate, radially weighted homogeneous semiholomorphic polynomial $\hat{f}$, such that the zeros of its $g$-polynomial form the braid $B_1$. By Lemmas~\ref{lem:symsemi} and \ref{lem:divsemi} we may choose the weight vector corresponding to $\hat{f}$ of the form $(k,1)$, where we may choose $k$ arbitrarily large. In particular, the weight vector associated with $\hat{f}$ may chosen to be greater than all weight vectors associated with 1-faces of $\tilde{f}$ with respect to the usual order of weight vectors. Since $\hat{f}$ is $u$-convenient, $\hat{f}_{\Delta'}$ has no zeros in $\mathbb{C}^*\times\mathbb{C}$, where $\Delta'$ is the vertex of the Newton boundary of $\hat{f}$ on the horizontal axis. In fact, since $\hat{f}$ is semiholomorphic, we have $\hat{f}_{\Delta'}(u,v)=cu^s$, where $s$ is the number of strands of $B_1$ and $c\in\mathbb{C}^*$.

Consider now the semiholomorphic polynomial
\begin{equation}
f(u,v)=\tilde{f}_{\Delta}(u,v)\hat{f}(u,v)+\hat{f}_{\Delta'}(u,v)\tilde{f}(u,v)-\tilde{f}_{\Delta}(u,v)\hat{f}_{\Delta'}(u,v).
\end{equation}

The function $f$ is semiholomorphic and has $N$ compact 1-faces in its Newton boundary. Let $P_i$, $i=1,2,\ldots,N$, denote the weight vector corresponding to these 1-faces and let $k_1=\tfrac{p_1}{p_2}$ with $P_1=(p_1,p_2)$. Note that $P_1$ is the weight vector that corresponds to the unique 1-face of $\hat{f}$ and $P_i$, $i=2,3,\ldots,N$, correspond to the 1-faces of $\tilde{f}$. Then for all weight vectors $P$ with $P_1\succ P$, the corresponding face function $f_P$ is the product of $\hat{f}_{\Delta'}$ which has no zeros in $(\mathbb{C}^*)^2$ and $\tilde{f}_{P}$, which is Newton non-degenerate. It follows from Lemma~\ref{lem:verttimes} that $f_{P}$ is Newton non-degenerate. For weight vectors $P$ with $P\succ P_1$ or $P=P_1$ the face function $f_{P}$ is the product of $\tilde{f}_{\Delta}$ and $\hat{f}_{P}$ and by the same argument as above, it is Newton non-degenerate. Since $f$ is also convenient, it is inner non-degenerate \cite{AraujoBodeSanchez}.

If $B_1$ contains a strand of the form $\{0\}\times S^1$, we have to multiply $f$ by $u$, which does not change inner non-degeneracy. Since the zeros of the face functions $f_{P_i}$ are exactly the zeros of $\tilde{f}_{P_i}$ together with the set $u=0$ and the zeros of $f_{P_1}$ are the zeros of $\hat{f}$ together with the set $v=0$, the link of the singularity of $f$ is $B(B_1,B_2,\ldots,B_N)$.

In order to obtain $B^o(B_1,B_2,\ldots,B_N)$ we multiply the polynomial that we constructed for $B(B_1,B_2,\ldots,B_N)$ by $v$. The result is again inner non-degenerate by Lemma~\ref{lem:uconv}.
\end{proof}


\section{Links of inner non-degenerate mixed functions}\label{sec5}

In this section, we prove Theorem~\ref{thm:main2}, which classifies the links of singularities of inner non-degenerate mixed polynomials with nice Newton boundary. This additional condition of being nice is already required in the result from \cite{AraujoBodeSanchez} that describes the links of inner non-degenerate mixed functions in terms of their face functions. For inner non-degenerate functions with semiholomorphic principal part it is automatically satisfied.
\begin{definition}\label{def:blabla}
We say that a mixed polynomial $f:\mathbb{C}^2\to\mathbb{C}$ is \textbf{nice} if for every non-extreme vertex $\Delta$ of $f$, we have $V_{f_{\Delta}}\cap (\C^*)^{2}=\emptyset$.
\end{definition}

\begin{lemma}\label{lem:miximplies}
Let $f:\mathbb{C}^2\to\mathbb{C}$ be an inner non-degenerate, mixed function with nice Newton boundary. Then the link of its weakly isolated singularity is of the form $L([L_1,L_2,\ldots,L_{N-1}],L_N)$ for some $N$ and some sequence of links $L_i$ where each $L_i$ is fixed-point-free even or odd.
\end{lemma}
\begin{proof}
The case where $f$ is exactly one compact 1-face in its Newton boundary is covered by Theorem~\ref{thm:main1}. Now assume that the Newton boundary of $f$ has more than $N'>1$ 1-faces.
As usual we write $P_i$, $i=1,2,\ldots,N'$, for the weight vectors associated to the 1-faces of $f$.
We know from \cite{AraujoBodeSanchez} that $f$ has a weakly isolated singularity, whose link is of the form $L([\widetilde{L_1},\widetilde{L_2},\ldots,\widetilde{L_{N'-1}}],\widetilde{L_{N'}})$, where each $\widetilde{L_i}\subset(\mathbb{C}\backslash\{0\}\times S^1)$, $i=2,\ldots,N'-1$, is given by the roots with non-zero $u$-coordinate of the $g$-polynomial associated with the face function $f_{P_i}$, $\widetilde{L_1}\subset\mathbb{C}\times S^1$ is given by the roots of the $g$-polynomial of $f_{P_1}$ and $\widetilde{L_{N'}}\subset S^1\times\mathbb{C}$ is given by the zeros of the $h$-polynomial associated with the face function $f_{P_{N'}}$. As pointed out in \cite{AraujoBodeSanchez}, some of the $\widetilde{L_i}$ might be the empty set, which of course possesses all the symmetries discussed in this article.

Since every face function is radially weighted homogeneous, all links must be even or odd by Lemma~\ref{lem1mix}. Since there are no fixed points of $\tau_u$ or $\tau_v$ in $(\mathbb{C}\backslash\{0\})\times S^1$, all links $\widetilde{L_i}$, $i=2,3,\ldots,N'-1$ are automatically fixed-point-free.

Write $P_i=(p_{i,1},p_{i,2})$ with $\gcd(p_{i,1},p_{i,2})$ and $i\in\{1,2,\ldots,N\}$. If $p_{1,1}$ is even, then $\widetilde{L_1}$ is $u$-even and it must be fixed-point-free, since there are no fixed points of $\tau_u$ in $\mathbb{C}\times S^1$. The analogous statement holds for $\widetilde{L_{N'}}$ if $p_{N',2}$ is even.

If $p_{1,2}$ is even, then $\widetilde{L_1}$ is $v$-even. By the same arguments as in Lemma~\ref{lem:convcases} it is fixed-point-free if $f$ is $v$-convenient. Otherwise, it is the union of a fixed-point-free $v$-even link, say $L_2$ and the set of fixed points of $\tau_v$. Since $L_2$ is fixed-point-free, there is a tubular neighbourhood $U$ of the set of fixed points that does not intersect $L_2$. We may deform the set of fixed points of $\tau_v$ in $U$ into an odd knot $L_1$. It follows from the construction in \cite{AraujoBodeSanchez} that $\widetilde{L_1}=[L_1,L_2]$ in $\mathbb{C}\times S^1$ and thus $L([\widetilde{L_1},\widetilde{L_2},\ldots,\widetilde{L_{N'-1}}],\widetilde{L_{N'}})=L([L_1,L_2,\widetilde{L_2},\ldots,\widetilde{L_{N'-1}}],\widetilde{L_{N'}})$. 

Similarly, if $p_{N',1}$ is even and $f$ is not $u$-convenient, the link $\widetilde{L_{N'}}$ is the union of a fixed-point-free $u$-even link $L_{-2}$ and the set of fixed points of $\tau_u$, which may be deformed into an odd knot, say $L_{-1}$, such that $L_{N'}=L(L_{-2},L_{-1})$ (see again \cite{AraujoBodeSanchez}). We then have $L([\widetilde{L_1},\widetilde{L_2},\ldots,\widetilde{L_{N'-1}}],\widetilde{L_{N'}})=L([\widetilde{L_1},\widetilde{L_2},\ldots,\widetilde{L_{N'-1}},L_{-2}],L_{-1})$ or $L([\widetilde{L_1},\widetilde{L_2},\ldots,\widetilde{L_{N'-1}}],\widetilde{L_{N'}})=L([L_1,L_2,\widetilde{L_2},\ldots,\widetilde{L_{N'-1}},L_{-2}],L_{-1})$. By setting $N=N'$, $N'+1$ or $N'+2$, depending on the parities of $p_{1,2}$ and $p_{N',1}$ and convenience of $f$, and by relabeling the links, we obtain the desired link $L([L_1,L_2,\ldots,L_{N-1}],L_N)$.
\end{proof}

\begin{lemma}\label{lem:vevenconstr}
Let $L$ be a $v$-even link in $(\mathbb{C}\backslash\{0\})\times S^1$. Then there exists a radially weighted homogeneous, convenient, Newton non-degenerate mixed polynomial $f:\mathbb{C}^2\to\mathbb{C}$ such that the zeros of the corresponding $g$-polynomial are isotopic to $L$ in $(\mathbb{C}\backslash\{0\})\times S^1$.
\end{lemma}
\begin{proof}
Suppose that $L$ is parametrised by $(u_j(\tau),\rme^{\rmi t_j(\tau)})$, where $j=1,2,\ldots,\ell$, is labelling the components of $L$ and $\tau$ goes from $0$ to $2\pi$. Then the link $L'$ given by $(\rme^{\rmi \arg (u_j(\tau))},\tfrac{1}{|u_j(\tau)|}\rme^{\rmi t_j(\tau)})\subset S^1\times(\mathbb{C}\backslash\{0\})$ is still $v$-even. By the analogue of Lemma~\ref{lem:symmix} for $v$-even links in $S^1\times(\mathbb{C}\backslash\{0\})$ there is a radially weighted homogeneous, convenient, Newton non-degenerate mixed polynomial $f$ such that the zeros of the corresponding $h$-polynomial are isotopic to $L'$ in $S^1\times(\mathbb{C}\backslash\{0\})$. It then follows from the definition of the $h$-polynomial and $g$-polynomial associated with $f$ that the zeros of $g$ are isotopic to $L$ in $(\mathbb{C}\backslash\{0\})\times S^1$.
\end{proof}

Naturally, the analogous result holds for $u$-even links in $S^1\times(\mathbb{C}\backslash\{0\})$.

\begin{lemma}\label{lem:Newtonconstr}
Let $L_1,L_2,\ldots,L_N$ be a sequence of links with $L_1\subset\mathbb{C}\times S^1$, $L_i\subset(\mathbb{C}\backslash\{0\})\times S^1$ for all $i\in\{2,3,\ldots,N-1\}$, $L_N\subset S^1\times\mathbb{C}$ and suppose that each of them is fixed-point-free even or odd. Then there exists a convenient, Newton non-degenerate mixed polynomial with nice Newton boundary whose link of the singularity is $L([L_1,L_2,\ldots,L_{N-1}],L_N)$ if $N>1$ and $L_1$ if $N=1$.
\end{lemma}
\begin{proof}
The case of $N=1$ is covered by the proof of Proposition~\ref{prop:okarad}. By Lemmas~\ref{lem:nonzerosym} and \ref{lem:nonzeroodd} we can assume that $L_1$ does not intersect $\{0\}\times S^1$ and $L_N$ does not intersect $S^1\times\{0\}$. Then by Lemma~\ref{lem:symmix} and Lemma~\ref{lem:vevenconstr} there is for every $i\in\{1,2,\ldots,N-1\}$ a radially weighted homogeneous, convenient, Newton non-degenerate mixed polynomial $f$ such that the zeros of the corresponding $g$-polynomial form the link $L_i$. Likewise, for $L_N$ there is such a polynomial such that the zeros of the corresponding $h$-polynomial form the link $L_N$.

The proof now follows exactly the same inductive construction as the proof of Lemma~\ref{lem:semiglue}. We obtain a mixed polynomial $f$ with $N$ compact 1-faces in its Newton boundary, such that the zeros of the $g$-polynomial corresponding to the $i$th face with non-vanishing $u$-coordinate form $L_i$ for all $i\in\{1,2,\ldots,N-1\}$ and the zeros of the $h$-polynomial corresponding to the last 1-face with non-vanishing $v$-coordinate form $L_N$.
Since none of the links intersect the sets defined by $u=0$ or $v=0$, the polynomial is convenient and all of the face functions corresponding to vertices of $f$ have no zeros in $(\mathbb{C}^*)^2$. In particular, $f$ has a nice Newton boundary and all of these face functions are Newton non-degenerate. Since all roots of the $g$-polynomials whose $u$-coordinate is non-zero are regular points, it follows that the face functions corresponding to 1-faces of $f$ are Newton non-degenerate as well. Thus $f$ is Newton non-degenerate. By \cite{AraujoBodeSanchez} the link of its singularity is $L([L_1,L_2,\ldots,L_{N-1}],L_N)$.
\end{proof}

Together with Lemma~\ref{lem:miximplies} this implies Theorem~\ref{thm:main2} and Theorem~\ref{cor:Oka}, since convenient, Newton non-degenerate polynomials are inner non-degenerate \cite{AraujoBodeSanchez}.

The construction in Lemma~\ref{lem:Newtonconstr} illustrates that a knot can only be realised as the link of the singularity of a convenient, inner non-degenerate mixed function with nice Newton boundary if it is fixed-point-free even or odd.
In particular, the knot $8_{16}$ from Example~\ref{ex:knot} cannot be realised in this way, nor as the link of the singularity of a convenient, Newton non-degenerate mixed function with nice Newton boundary. Since $8_{16}$ is fibered, this also shows that in order to prove the Benedetti-Shiota conjecture \cite{benedetti}, degenerate mixed polynomials must be considered.

\section{P-fibered braids with multiplicities and coefficients}\label{sec6}

Now we turn our attention to isolated singularities. Since generic mixed polynomials do not have isolated singularities, constructions are much more challenging than in the case of weakly isolated singularities. By restricting ourselves to strongly inner non-degenerate semiholomorphic polynomials, we can guarantee that the resulting polynomials have isolated singularities \cite{AraujoBodeSanchez}, which allows us to describe the links of the singularities of this family of functions.

\begin{definition}\cite[Definition 6.1]{AraujoBodeSanchez}\label{def:stronginner}
We say that a mixed polynomial $f:\mathbb{C}^2\to\mathbb{C}$ is \textbf{strongly inner non-degenerate} if both of the following conditions hold:
\begin{enumerate}
\item[(i)] the face functions $f_{P_1}$ and $f_{P_N}$ have no critical points in $\mathbb{C}^2\backslash\{v=0\}$ and $\mathbb{C}^2\backslash\{u=0\}$, respectively.
\item[(ii)] for each 1-face and non-extreme vertex $\Delta$, the face function $f_{\Delta}$ has no critical points in $(\mathbb{C}^*)^2$.
\end{enumerate}
\end{definition}

If $f$ is radially weighted homogeneous, $u$-convenient and semiholomorphic, we may interpret the corresponding $g$-polynomial as a loop in the space of polynomials in one complex variable $u$ and with fixed degree. The coefficients of $g$ (as a polynomial in $u$) are polynomials in $\rme^{\rmi t}$ and $\rme^{-\rmi t}$, so that the angular coordinate $t$ of the second complex variable $v$ parametrises the loop of polynomials. The function $f$ is inner non-degenerate if and only if the roots of $g(\cdot,\rme^{\rmi t})$ are simple (i.e., distinct) for all $t\in[0,2\pi]$. In this case, the zeros of $g$ form a (closed) braid in $\mathbb{C}\times S^1$ on $s=\deg_ug=\deg_u f$ strands and parametrised by $t\in[0,2\pi]$, i.e., 
\begin{equation}
g(u,\rme^{\rmi t})=c\prod_{j=1}^s(u-u_j(t))
\end{equation} 
for appropriate functions $u_j:[0,2\pi]\to\mathbb{C}$ and $c\in\mathbb{C}^*$.

Conversely, any braid given as a set of parametric curves 
\begin{equation}\label{eq:braidpara}
\bigcup_{j=1}^s\bigcup{t\in[0,2\pi]}(u_j(t),\rme^{\rmi t})\subset\mathbb{C}\times S^1
\end{equation} 
can be interpreted as a loop in the space of monic polynomials in one complex variable and degree $s$, given by $\prod_{j=1}^s(u-u_j(t))$. If we want the functions $u_j(t)$ to parametrise the strands of a braid they need to satisfy that $u_i(t)\neq u_j(t)$ for all $t\in[0,2\pi]$ if $i\neq j$ and for every $i\in\{1,2,\ldots,s\}$ there exists a $j\in\{1,2,\ldots,s\}$ with $u_i(0)=u_j(2\pi)$.

It is sometimes necessary to distinguish between a braid given as a set of parametric curves and its isotopy class, although both are often referred to simply as a ``braid''. If we want to emphasize that we are dealing with a particular representative of a braid isotopy class we refer to it as a \textit{geometric braid}.

\begin{definition}
A geometric braid $B$ parametrised by $\cup_{j=1}^s\cup{t\in[0,2\pi]}(u_j(t),\rme^{\rmi t})\subset\mathbb{C}\times S^1$ is called P-fibered if $\arg g:(\mathbb{C}\times S^1)\backslash B\to S^1$ is a fibration map, where $g:\mathbb{C}\times S^1\to\mathbb{C}$, $g(u,\rme^{\rmi t})=\prod_{j=1}^s(u-u_j(t))$ is the loop of polynomials corresponding to $B$.

We call a braid isotopy class P-fibered if it contains a P-fibered geometric braid.
\end{definition}

Note that the property of being P-fibered is a geometric property, that is, we may deform P-fibered geometric braids into geometric braids that are not P-fibered. P-fibered braids close to fibered links in $S^3$ \cite{bode:sat} and have been central in various constructions of isolated singularities \cite{bode:ralg, bode:sat, bode:thomo}. 

\begin{proposition}\label{prop:radPfib}
A radially weighted homogeneous, $u$-convenient and semiholomorphic polynomial $f$ is strongly inner non-degenerate if and only if the zeros of the corresponding $g$-polynomial form a P-fibered braid.
\end{proposition}
\begin{proof}
Since $f$ is radially weighted homogeneous, it is strongly inner non-degenerate if and only if it has an isolated singularity at the origin \cite{eder2}. The semiholomorphicity of $f$ implies that every critical point of $f$ must be a zero of $\tfrac{\partial f}{\partial u}$. Let $P=(p_1,p_2)$ be the weight vector of $f$ and let $k=\tfrac{p_1}{p_2}$. Then the zeros of $f$ are parametrised by $(r^ku_j(t),r\rme^{\rmi t})\subset\mathbb{C}^2$, where $u_j(t)$, $j=1,2,\ldots,s$, is a parametrisation of the roots of the corresponding $g$-polynomial $g$. Similarly, the roots of $\tfrac{\partial f}{\partial u}$ are parametrised by $(r^kc_j(t),r\rme^{\rmi t})\subset\mathbb{C}^2$, where $c_j(t)$, $j=1,2,\ldots,s-1$, is a parametrisation of the roots of $\tfrac{\partial g}{\partial u}$. This means that $f$ has a weakly isolated singularity if and only if the roots of $g(\cdot,t)$ are distinct for all $t\in[0,2\pi]$.

We know that $f$ is $u$-convenient and so $f(u,0)=cu^s$ for some $c\in\mathbb{C}^*$. In particular, there are no critical points of $f$ of the form $(u,0)\in\mathbb{C}^2$. Consider the matrix of first derivatives of $(|f|,\arg(f))$ at $p\in\mathbb{R}^4$ in the basis $(\partial_{\text{Re}(u)},\partial_{\text{Im}(u)},\partial_{\arg(v)}, w)$ of $T_p\mathbb{R}^4$, where $w=kr^{k-1}\partial_{|u|}+\partial{|v|}$ and $p=(\text{Re}(u),\text{Im}(u),r\cos(\arg(v)),r\sin(\arg(v)))$. Taking $p$ to be of the form $(r^k\text{Re}(c_j(t)),r^k\text{Im}(c_j(t)),r\cos(t),r\sin(t))$ with $r>0$, we find that $(w|f|)(p)\neq 0$, but $(w\arg(f))(p)=0$. Since all derivatives with respect to the real or imaginary part of $u$ vanish at $p$, it follows that $p$ is a critical point if and only if $\tfrac{\partial \arg(f)}{\partial t}(p)=0$, which happens if and only if $\tfrac{\partial \arg(g)}{\partial t}(c_j(t),t)=0$. Thus the origin is the only critical point of $f$ if and only if $\arg(g)$ has no critical points, in other words, the roots of $g$ form a P-fibered braid.
\end{proof}

We may combine this with the proof of Theorem~\ref{thm:main1b} and obtain the following.
\begin{corollary}
A link type $L\subset S^3$ arises as the link of an isolated singularity of some radially weighted homogeneous, $u$-convenient semiholomorphic polynomial $f$ if and only if there is a P-fibered geometric braid $B\subset \mathbb{C}\times S^1$ that is $u$-even or divisor-symmetric and whose closure is $L$.
\end{corollary}
\begin{proof}
The corollary is a combination of the proof of Proposition~\ref{prop:radPfib} and the proofs in Section~\ref{sec2}, in particular the proof of Lemmas~\ref{lem:symsemi} and \ref{lem:divsemi}. The only important additional ingredient is that if the roots of loop of polynomials $g_t$ form a P-fibered geometric braid, then the same is true for any sufficiently $C^1$-close approximation of $g_t$.
\end{proof}

The following corollary is then a direct consequence of the work in \cite{AraujoBodeSanchez}.
\begin{corollary}
A link type $L\subset S^3$ arises as the link of an isolated singularity of some strongly inner non-degenerate $u$-convenient polynomial $f$ with exactly one compact 1-face in its Newton boundary and semiholomorphic principal part if and only if there is a P-fibered geometric braid $B\subset \mathbb{C}\times S^1$ that is $u$-even or divisor-symmetric and whose closure is $L$.
\end{corollary}

We would like to generalise these results to polynomials with more than one 1-face in its Newton boundary. In order to do this, we need to generalise the notion of a P-fibered braid.

\begin{definition}
Let $B$ be a geometric braid on $s$ strands, parametrised as in Eq.~\eqref{eq:braidpara}, but with the property that $u_j(t)\neq 0$ for all $t\in[0,2\pi]$. That is, none of the strands intersects the line $\{0\}\times[0,2\pi]\subset\C\times[0,2\pi]$. Let $g(\cdot,\rme^{\rmi t})$ be the corresponding loop of polynomials given by $g(u,\rme^{\rmi t})=\prod_{j=1}^s(u-u_j(t))$. Then we say that $B$ is \textbf{P-fibered with $O$-multiplicity $m$ and coefficient $a(t)$} if $\arg(a(t)u^m g):(\C\times S^1)\backslash(B\cup(\{0\}\times S^1))\to S^1$ is a fibration map, i.e., if it has no critical points, for some nowhere-vanishing finite Fourier series $a:S^1\to\C^*$ written as a polynomial in $\rme^{\rmi t}$ and $\rme^{-\rmi t}$.
\end{definition}

P-fibered braids of $O$-multiplicity 0 and constant coefficient are P-fibered braids. If $B$ is a P-fibered braid of $O$-multiplicity 1 and constant coefficient, then $B\cup(\{0\}\times[0,2\pi])$ is a P-fibered braid. If the $O$-multiplicity $m$ is greater than 1 or if $a(t)$ is not constant, then the closure of $B\cup(\{0\}\times[0,2\pi])$ is not necessarily a fibered link, since in this case $\arg(u^m g)$ does not take the required form on a tubular neighbourhood of $\{0\}\times S^1$ or on the boundary of $\C\times S^1$. The following lemma illustrates this point.
\begin{lemma}\label{lem:a}
For every geometric braid $B$ and every non-negative integer $m$ there exists a non-vanishing finite Fourier series $a:S^1\to\C^*$ such that $B$ is P-fibered with $O$-multiplicity $m$ and coefficient $a(t)$.
\end{lemma}
\begin{proof}
Let $v_{j}(t)$, $j=1,2,\ldots,s-1$, denote the non-zero critical values of $u\mapsto u^mg(u,\rme^{\rmi t})$, $t\in[0,2\pi]$, where $g(\cdot,\rme^{\rmi t})$ is the loop of polynomials corresponding to the geometric braid $B$ and $s$ is the number of its strands. Then the non-zero critical values of $a(t)u^mg$ are precisely $a(t)v_j(t)$. A critical point of $\arg(a(t)u^mg)$ is a solution to the equation $\tfrac{\partial\arg(a(t)v_j(t))}{\partial t}=\tfrac{\partial \arg(a(t))}{\partial t}+\tfrac{\partial\arg(v_j)}{\partial t}=0$ \cite{bode:sat}. Hence by choosing $a$ such that $\tfrac{\partial \arg(a(t))}{\partial t}$ has sufficiently large modulus, there are no argument-critical points of $\arg(a(t)u^mg)$.
\end{proof}

If a geometric braid is P-fibered with $O$-multiplicity $m$ and constant leading coefficient, we often simply say that it is P-fibered with $O$-multiplicity $m$.

Note that, similarly to P-fibered braids, a geometric braid $B$ is P-fibered with $O$-multiplicity $m$ and coefficient $a(t)$, if and only if $B^n$ is P-fibered with $O$-multiplicity $m$ and coefficient $a(nt)$ for all $n\in\mathbb{Z}\backslash\{0\}$. 

\begin{lemma}\label{lem:approxcomp}
Let $B$ be a P-fibered geometric braid with $O$-multiplicity $m$ and coefficient $a(t)$. Then for any geometric braid $B'$ that is a sufficiently close $C^1$-approximation of $B$ and any smooth $2\pi$-periodic function $\tilde{a}(t)$ that is a sufficiently close $C^1$-approximation of $a(t)$ the geometric braid $B'$ is P-fibered with $O$-multiplicity $m$ and coefficient $\tilde{a}(t)$.
\end{lemma}
\begin{proof}
The critical values of a complex polynomial are smooth functions of its coefficients, which are themselves smooth functions of its roots. So if $\arg(a(t)u^mg(u,\rme^{\rmi t}))$ has no critical points, then neither does $\arg(\tilde{a}(t)u^m\tilde{g}(u,\rme^{\rmi t}))$ for any sufficiently close $C^1$-approximation of $a(t)$ and $g$, the latter corresponding to an approximation of the geometric braid $B$.
\end{proof}

\begin{definition}\label{def:compseq}
Let $B_1, B_2, \ldots,B_N$ be a finite sequence of geometric braids with $B_i\subset(\mathbb{C}\backslash\{0\})\times[0,2\pi]$ for all $i>1$. Suppose that $B_i$ consists of $s_i$ strands and is P-fibered with $O$-multiplicity $m_i:=\sum_{j<i}s_j$ and coefficient $a_i(t)$ for all $i\in\{1,2,\ldots,N\}$. Let $\widetilde{g_i}$ be the monic polynomial whose roots form $B_i$ and let $b_i(t)$ denote its lowest order coefficient. We call such a sequence of geometric braids a \textbf{compatible sequence} if
\begin{itemize}
\item $a_N(t)=1$, 
\item $a_{i-1}(t)=b_i(t)a_i(t)=\prod_{j=i}^Nb_i$ for all $i\in\{2,3,\ldots,N\}$.
\end{itemize}
\end{definition}



\begin{lemma}\label{lem:isobraid}
Let $f$ be a $u$-convenient, strongly inner non-degenerate mixed polynomial with semiholomorphic principal part and exactly $N$ compact 1-faces. Then the link of its isolated singularity is the closure of a braid of the form $B(B_1,B_2,\ldots,B_N)$ for some compatible sequence of geometric braids $B_1, B_2,\ldots,B_N$, all of which are $u$-even or divisor-symmetric.
\end{lemma}
\begin{remark}
Note that, just like P-fiberedness, the defining property of a compatible sequence is one of geometric braids. In particular, in Lemma~\ref{lem:isobraid}  the same geometric braids $B_i$ that form a compatible sequence have to satisfy the symmetry constraint. It is not enough for a $B_i$ to be braid isotopic to a $u$-even or divisor-symmetric geometric braid.
\end{remark}
\begin{proof}
By \cite{AraujoBodeSanchez} the link of the singularity is determined by the principal part of $f$, so that we may assume that $f$ is equal to its principal part and thus semiholomorphic. Let $f_{P_i}$ denote the $i$th compact 1-face of $f$ and let $g_i$ be the $g$-polynomial associated with $f_{P_i}$.

Furthermore, the link of the singularity is the closure of a braid of the form $B(B_1,B_2,\ldots,B_N)$ for a sequence of geometric braids with $B_i\subset(\mathbb{C}\backslash\{0\})\times S^1$, which are determined by the roots of $g_i$ with non-zero $u$-coordinate. By Theorem~\ref{thm:main1b} all of these geometric braids are $u$-even or divisor-symmetric.

Since $f$ is $u$-convenient, its highest order term with respect to $u$ is of the form $cu^s$ for some $c\in\mathbb{C}\backslash\{0\}$. By the same arguments as in Proposition~\ref{prop:radPfib} the fact that $f_{P_N}$ has no critical points in $\mathbb{C}^*\times\mathbb{C}$ is equivalent to $\arg(g_N)$ having no critical points. By definition we have that $g_N=cu^{m_N}\widetilde{g_N}$, where $\widetilde{g_N}(\cdot,t)$ is the monic loop of polynomials whose roots form $B_N$, $m_N=\sum_{j<N}s_j$ and $s_i$ is the number of strands of $B_i$. If $\arg(cu^{m_N}\widetilde{g_N})$ has no critical points, then $B_N$ is a P-fibered geometric braid with $O$-multiplicity $m_N$ and coefficient $a_N(t)=c$. Since $c$ is a constant, it is also P-fibered with the same $O$-multiplicity and leading coefficient $a_N(t)=1$.

Similarly, the arguments in Proposition~\ref{prop:radPfib} imply that $\arg(g_i)$ has no critical points. The only difference to the proof of Proposition~\ref{prop:radPfib} is that $f_{P_i}$ is not $u$-convenient if $i<N$. This however does not affect the argument for points in $(\mathbb{C}^*)^2$. Thus $B_i$ is P-fibered with $O$-multiplicity $m_i=\sum_{j=1}^{i-1}s_j$ and coefficient equal to the coefficient $a_i(t)$ of highest order term of $g_i$ with respect to $u$. This highest order comes from the vertex on the Newton boundary that is shared by the 1-face of $f_{P_i}$ and $f_{P_{i+1}}$. It is therefore equal to the lowest order term of $g_{i+1}$, which is precisely $b_{i+1}(t)a_{i+1}(t)$. It follows that the $B_i$ form a compatible sequence.
\end{proof}

\begin{lemma}
Let $B_1,B_2,\ldots,B_N$ form a compatible sequence of geometric braids, each of which is $u$-even or divisor-symmetric. Then there is a strongly inner non-degenerate mixed polynomial $f:\mathbb{C}^2\to\mathbb{C}$ with semiholomorphic principal part and exactly $N$ compact 1-faces such that the link of its isolated singularity is the closure of $B(B_1,B_2,\ldots,B_N)$.
\end{lemma}
\begin{proof}
Compatible sequences with $N=1$ are simply P-fibered geometric braids. Let $B$ be a P-fibered geometric braid that $u$-even or divisor-symmetric. Then we may construct a radially weighted homogeneous, $u$-convenient semiholomorphic polynomial $f:\mathbb{C}^2\to\mathbb{C}$ with a weakly isolated singularity and the link of the singularity is the closure of $B$ as in Theorem~\ref{thm:main1b}. In this construction in the proof of Theorem~\ref{thm:main1b} the $g$-polynomial associated with $f$ is an arbitrarily close $C^1$-approximation of the loop of polynomials $\widetilde{g}(\cdot,t)$ whose roots form the geometric braid $B$. Since P-fiberedness is an open condition, the roots of $g$ form a P-fibered geometric braid. By Proposition~\ref{prop:radPfib} $f$ is strongly inner non-degenerate. 

Now suppose that $N>1$. We want to proceed as in the proofs of Theorem~\ref{thm:main1b} and Theorem~\ref{thm:main2b}, that is, we need to approximate the loops of monic polynomials $\widetilde{g_i}(\cdot,t)$, whose roots form the geometric braids $B_i$, by loops of monic polynomials $g_i(\cdot,t)$ that still display the same type of symmetries, but whose coefficients are given by polynomials in $\rme^{\rmi t}$ and $\rme^{-\rmi t}$. As in Theorem~\ref{thm:main1b} this can be achieved via arbitrarily close $C^1$-approximations of the polynomials $\widetilde{g_i}$, or equivalently, $C^1$-approximations of the geometric braids $B_i$. The roots of the resulting polynomials $g_i$ are then braid isotopic to the $B_i$ in $\mathbb{C}^*\times S^1$ and by Lemma~\ref{lem:approxcomp} they still form a compatible sequence, where the leading coefficient of the braid $B_{i-1}$ is $\tilde{a}_{i-1}(t)$, the lowest term of $g_i$ multiplies the leading coefficient of $B_i$.

From each of these $\tilde{a}_i(t)g_i(u,\rme^{\rmi t})$ we obtain a radially weighted semiholomorphic polynomial $f_{P_i}$ in the usual way, where there is some choice in the weight vector of $f_{P_i}$. In particular, we can order them as in the proof of Theorem~\ref{thm:main2b}. Definition~\ref{def:compseq} ensures that the highest order term of $f_{P_{i-1}}$ with respect to $u$ is equal to the lowest order term of $f_{P_i}$ with respect to $u$. We thus have one semiholomorphic polynomial $f$, whose face functions corresponding to 1-faces are precisely the functions $f_{P_i}$, $i=1,2,\ldots,N$.

We now need to show that $f$ is strongly inner non-degenerate. Since the leading coefficients $a_i(t)$ are nowhere-vanishing, so are their approximations. This implies that the face functions of vertices of the Newton boundary have no critical points in $(\mathbb{C}^*)^2$. For the face functions $f_{P_i}$ we use a similar argument as in the proof of Proposition~\ref{prop:radPfib}. Any critical point of $f_{P_i}$ in $\mathbb{C}\times\mathbb{C}^*$ corresponds to a critical point of $\arg(u^{m_i}\tilde{a}_i(t)g_i(u,\rme^{\rmi t}))$. Since the roots of $g_i$ form by assumption a P-fibered braid with $O$-multiplicity $m_i$ and leading coefficient $\tilde{a}_i(t)$, there no such critical points. This shows the second property of Definition~\ref{def:stronginner} and the first property for $f_{P_1}$.

For the first condition of Definition~\ref{def:stronginner} for $f_{P_N}$ recall that $\tfrac{\partial f_{P_N}}{\partial u}(u,0)=su^{s-1}$ with $s=m_N+s_N$. So $f_{P_N}$ has no critical points in $\mathbb{C}^*\times\{0\}$ and so $f$ is strongly non-degenerate. 

By \cite{AraujoBodeSanchez} $f$ has an isolated singularity and its link is the closure of $B(B_1,B_2,\ldots,B_N)$.
\end{proof}

This concludes the proof of Theorem~\ref{thm:main3}. Note that in the construction we can also achieve that $B_1$ does not intersect $\{0\}\times S^1$. The constructed semiholomorphic polynomial $f$ can therefore be assumed to be convenient. It is strongly Newton non-degenerate if $\tfrac{\partial \arg(b_1(t)a_1(t))}{\partial t}$ never vanishes, where $b_1(t)$ is the lowest order coefficient of the loop of monic polynomials defining $B_1$ and $a_1(t)$ is the leading coefficient of $B_1$.

\begin{proposition}
A link type $L$ arises as the link of an isolated singularity of a convenient, strongly Newton non-degenerate mixed polynomial $f:\C^2\to\C$ with $N$ compact 1-faces and semiholomorphic principal part if and only if it has a representative that is the closure of $B(B_1,B_2,\ldots,B_N)$, where $B_1,B_2,\ldots,B_N$ is a compatible sequence of P-fibered braids with $O$-multiplicities and coefficients such that every braid $B_i$ is $u$-even or divisor-symmetric and $\tfrac{\partial \arg(b_1(t)a_1(t))}{\partial t}$ never vanishes, where $b_1(t)$ is the lowest order coefficient of the loop of monic polynomials defining $B_1$ and $a_1(t)$ is the leading coefficient of $B_1$.
\end{proposition}

The concept of a compatible sequence can be understood as a new way to construct P-fibered braids.
\begin{proposition}\label{prop:seqpfib}
Let $B_1,B_2,\ldots,B_N$ be a compatible sequence of geometric braids with $O$-multiplicities and coefficients. Then $B(B_1,B_2,\ldots,B_N)$ is a P-fibered braid.
\end{proposition}
\begin{proof}
The braid $B(B_1,B_2)$ is obtained from $B_1$ and $B_2$ as follows. We take the copy of $\mathbb{C}\times[0,2\pi]$ that contains $B_2$ and remove a tubular neighbourhood $U_1$ of $\{0\}\times[0,2\pi]$ from it. Since $B_2$ does not intersect $\{0\}\times[0,2\pi]$, we can choose $U_1$ such that it does not intersect $B_2$. Now glue the copy of $\mathbb{C}\times[0,2\pi]$ that contains $B_1$ to $(\mathbb{C}\times[0,2\pi])\backslash U_1$, identifying their boundaries $S^1\times[0,2\pi]$ without any twisting. For higher values of $N$ we use $B(B_1,B_2,\ldots,B_{N-1},B_N)=B(B_1,B(B_2,\ldots,B_{N-1},B_N))$, writing $U_i$, $i=1,2,\ldots,N-1$, for the corresponding tubular neighbourhood of $\{0\}\times[0,2\pi]$ in the copy of $\mathbb{C}\times[0,2\pi]$ that contains $B_{i+1}$.

Let now $B_1,B_2,\ldots,B_N$ be a compatible sequence. Let $\widetilde{g_i}$ be the corresponding loop of monic polynomials whose roots are $B_i$. Then $\widetilde{g_i}(0,\rme^{\rmi t})=b_i(t)$ and so 
\begin{equation}\label{eq:limit1}
\lim_{R\to 0}\arg((R\rme^{\rmi \chi})^{m_i}a_i(t)\widetilde{g_i}(R\rme^{\rmi\chi}),\rme^{\rmi t})=m_i\chi+\arg(b_i(t))+\arg(a_i(t)).
\end{equation}
This determines the behaviour of $\arg((R\rme^{\rmi \chi})^{m_i}\widetilde{g_i}(R\rme^{\rmi\chi}),\rme^{\rmi t})$ on the boundary of $U_{i-1}$.

On the other hand, we have
\begin{align}\label{eq:limit2}
\lim_{R\to\infty}\arg(R\rme^{\rmi \chi})^{m_{i-1}}\widetilde{g_{i-1}}(R\rme^{\rmi \chi},\rme^{\rmi t}))&=m_i\chi+\arg(a_{i-1}(t))\nonumber\\
=m_i\chi+\arg(b_i(t))+\arg(a_i(t)),
\end{align}
where the last equality follows from the second condition in Definition~\ref{def:compseq}.

Using the closed disk as a compactification of $\mathbb{C}$, this describes the behaviour of $\arg(R\rme^{\rmi \chi})^{m_{i-1}}\widetilde{g_{i-1}}(R\rme^{\rmi \chi},\rme^{\rmi t}))$ on the boundary of the solid cylinder that is glued to the boundary of $U_{i-1}$ in $(\mathbb{C}\times[0,2\pi])\backslash U_{i-1}$. Let $X_i=(\mathbb{C}\times S^1)\backslash(B_i\cup U_{i-1})$ for $i\in\{2,3,\ldots,N\}$ and $X_1=(\mathbb{C}\times S^1)\backslash B_1$. Since Eq.~\eqref{eq:limit1} and Eq.~\eqref{eq:limit2} match for all $i$, we obtain a well-defined circle-valued map $\Phi:(\mathbb{C}\times S^1)\backslash B(B_1,B_2,\ldots,B_N)\to S^1$ by setting $\Phi(u,\rme^{\rmi t})=\arg(u^{m_i}a_i(t)\widetilde{g_i}(u,\rme^{\rmi t}))$ on $X_i$.

By definition $\arg((R\rme^{\rmi \chi})^{m_i}a_i(t)\widetilde{g_i}(R\rme^{\rmi\chi}),\rme^{\rmi t})$ has no critical points and thus $\Phi$ is a fibration on $(\mathbb{C}\times S^1)\backslash B(B_1,...,B_N)$ over $S^1$. Since $B_1$ has leading coefficient $a_1(t)=1$, we have
\begin{equation}\label{eq:limit3}
\lim_{R\to\infty}\arg((R\rme^{\rmi \chi})^{m_N}a_N(t)\widetilde{g_N}(R\rme^{\rmi\chi}),\rme^{\rmi t})=(m_N+s_N)\chi.
\end{equation}

As in \cite{bode:braided} we can turn $\Phi$ into a fibration on the complement of the closure of $B(B_1,B_2,\ldots,B_N)$ in $S^3$ by identifying points on the same longitude on the torus boundary of $(\mathbb{C}\times S^1)$. The torus boundary is thus collapsed to an unknot $O$, corresponding to a meridian of the boundary of $\mathbb{C}\times S^1$, and thus a braid axis for $B(B_1,B_2,\ldots,B_N)$. By Eq.~\eqref{eq:limit3}, this unknot is positively transverse to the fibers of $\Phi$ and thus by definition $O$ and the closure of $B(B_1,B_2,\ldots,B_N)$ are generalised exchangeable, see \cite{bode:braided}. It follows from Theorem 1.1 in \cite{bode:braided} that $B(B_1,B_2,\ldots,B_N)$ is a P-fibered braid.
\end{proof}

If every $B_i$ in a compatible sequence is $u$-even, then $B(B_1,B_2,\ldots,B_N)$ is $u$-even itself. Thus the closure of $B(B_1,B_2,\ldots,B_N)$ is real algebraic and the corresponding semiholomorphic polynomial can be taken to be radially weighted homogeneous \cite{bode:sat}. In particular, all examples in Section 6 in \cite{AraujoBodeSanchez} can be obtained in this way. The advantage of working with compatible sequences is that the $B_i$s can satisfy different symmetries. Some of them can be $u$-even and others can be odd, so that $B(B_1,B_2,\ldots,B_N)$ does not satisfy any particular symmetry. However, its closure can nonetheless be constructed as a real algebraic link.

For a given braid it is not easy to determine for which $O$-multiplicities and which leading coefficients it can be represented by a geometric braid as a P-fibered braid with the given data. However, we can still construct several families of compatible sequences.
\begin{lemma}\label{lem:Omult}
Let $B$ be a P-fibered geometric braid of $O$-multiplicity $m$. Then for any leading coefficient $a(t)$ there exists an integer $q$ such that for all $p>q$ the geometric braid $B^p$ is P-fibered with $O$-multiplicity $m$ and leading coefficient $a(t)$. 
\end{lemma}
\begin{proof}
Let $v_{j}(t)$, $j=1,2,\ldots,s-1$, denote the non-zero critical values of $u^mg(\cdot,\rme^{\rmi t})$, where $g$ is the loop of polynomials corresponding to the geometric braid $B$ and $s$ is the number of its strands. Then the non-zero critical values of $a(t)u^mg(u,\rme^{p\rmi t})$ are precisely $a(t)v_j(pt)$. A critical point of $\arg(a(t)u^mg(u,\rme^{\rmi pt}))$ is a solution to the equation $\tfrac{\partial\arg(a(t)v_j(pt))}{\partial t}=\tfrac{\partial \arg(a)}{\partial t}+p\tfrac{\partial\arg(v_j)}{\partial t}=0$. Since $B$ is P-fibered with $O$-multiplicity $m$, we have that $\tfrac{\partial\arg(v_j(t))}{\partial t}$ is nowhere-vanishing for all $j=1,2,\ldots,s-1$. Therefore, by choosing $p$ sufficiently large, we can avoid any argument-critical points of $a(t)u^mg(u,\rme^{\rmi pt})$ and thus $B^p$ is P-fibered with $O$-multiplicity $m$ and leading coefficient $a(t)$ as long as $p$ is larger than some $q$.
\end{proof}

\begin{theorem}\label{thm:Pmult}
Let $B_i$, $i=1,2,\ldots,N$, be a sequence of geometric braids on $s_i$ strands that do not intersect $\{0\}\times[0,2\pi]$. Let $m_i=\sum_{j<i}s_j$ and suppose that $B_i$ is P-fibered with $O$-multiplicity $m_i$ for all $i\in\{1,2,\ldots,N\}$. Then for any sequence of positive integers $r_j, r_{j+1}, \ldots,r_{N}$ there is a positive integer $M_{j-1}$ such that $B_1^{r_1},B_2^{r_2},\ldots,B_{N-1}^{r_{n-1}},B_N^{r_N}$ is a compatible sequence if $r_j>M_j$ for all $j=1,2,\ldots,N-1$. It follows that the closure of $B(B_1^{2r_1},B_2^{2r_2},\ldots,B_{N-1}^{2r_{n-1}},B_N^{2r_N})$ is real algebraic.
\end{theorem}
\begin{proof}
Let $g_N(\cdot,\rme^{\rmi t})$ denote the loop of monic polynomials that realises $B_N$ as a P-fibered geometric braid with $O$-multiplicity $m_N$. Let $b_N(t)$ be the lowest order coefficient of $g_N(u,\rme^{\rmi t})$ with respect to $u$. Then $b_N(r_Nt)$ is the lowest order coefficient of the loop of monic polynomials that realises $B_N^{r_N}$ as a P-fibered geometric braid with $O$-multiplicity $m_N$. By Lemma~\ref{lem:Omult} there is now a positive integer $M_{N-1}$ such that $B_{N-1}^{r_{N-1}}$ is a P-fibered geometric braid with $O$-multiplicity $m_{N-1}$ and leading coefficient $b_N(r_Nt)$ for all $r_{N-1}>M_{N-1}$. Note however that the value of $M_{N-1}$ depends on the choice of $r_N$. Continuing like this inductively, we show that $B_i^{r_i}$ is a P-fibered geometric braid with $O$-multiplicity $m_i$ and leading coefficient $\prod_{j>i}b_j(r_jt)$, where $b_j(t)$ is the lowest order coefficient of the loop of monic polynomials $g_j(u,\rme^{\rmi t})$ whose roots form the geometric braid $B_j$. In other words, the geometric braids form a compatible sequence. If we take all of the exponents of the braids to be even, they are all $u$-even. Thus by Theorem~\ref{thm:main3} the closure of $B(B_1^{2r_1},B_2^{2r_2},\ldots,B_{N-1}^{2r_{n-1}},B_N^{2r_N})$ is real algebraic if $r_j>M_j$ for all $j<N$.
\end{proof}

Theorem~\ref{thm:Pmult} is reminiscent of the satellite constructions of P-fibered braids (and hence real algebraic links) in \cite{bode:sat} and the construction of real algebraic links from P-fibered braids with $O$-multiplicities in \cite{AraujoBodeSanchez}. In fact, generalizing the concept of P-fibered braids with $O$-multiplicity even further allows us to see all of these constructions as special cases of the same idea.

Without going into too much detail let us first point out some differences and similarities between the construction outlined above and that described in \cite{bode:sat}. In both cases, we start with a braid that is in some sense associated to a polynomial fibration map. In the construction above with $N=2$, this initial braid is the union of $B_2$, which is P-fibered with some non-zero $O$-multiplicity, and the 0-strand $0\times[0,2\pi]$. A new braid is then constructed by replacing a component of its closure, the 0-strand, by another braid $B_1$ that is P-fibered. In order to guarantee that this new braid $B(B_1,B_2)$ leads to an isolated singularity we have to use a certain power of $B_1$, say $B_1^{r_1}$ instead of $B_1$ itself.

In \cite{bode:sat} we start with a P-fibered braid, whose closure consists of $n$ components. Then for every component of the closure we choose a P-fibered geometric braid $B_i$, $i=1,2,\ldots,n$, on $s$ strands. Note that all $B_i$ have the same number of strands. We perform a satellite operation that replaces a tubular neighbourhood of each component $C_i$ with a solid torus containing the closed braid $B_i^{r_i}$ for some exponent $r_i$. Again, if the exponent $r_i$ is chosen sufficiently large, the resulting braid is P-fibered and hence the closure of its square is real algebraic.

While the construction in Theorem~\ref{thm:Pmult} replaces only one component, the construction in \cite{bode:sat} replaces every component by a satellite link of the original component. Since the resulting braid, denoted $\mathcal{B}(B;B_1^{r_1},B_2^{r_2},\ldots,B_n^{r_n})$ is P-fibered, the corresponding semiholomorphic polynomial is radially weighted homogeneous. By Proposition~\ref{prop:seqpfib} the same is true for $B(B_1^{2r_1},B_2^{2r_2},\ldots,B_N^{2r_N})$. However, Theorem~\ref{thm:Pmult} together with Theorem~\ref{thm:main3} also allows us to to realise the latter braid via semiholomorphic polynomials with $N$ compact 1-faces in its Newton boundary, which is presumably not possible for the former.

Another difference lies in the range of values of the different $r_j$ that are sufficient to guarantee the result. In Theorem~\ref{thm:Pmult} the sufficient range of values for the exponent $r_j$ depends on the chosen value of $r_{j+1}$, while in \cite{bode:sat} the numbers can be chosen independently from each other.

We generalise both constructions as follows.
\begin{definition}
Let $B$ be a geometric braid on $s$ strands with a closure that has $n$ components, labeled $C_i$, $i=1,2,\ldots,n$, and parametrised as
\begin{equation}
\bigcup_{i=1}^n\bigcup_{j=1}^{n_i}\bigcup_{t\in[0,2\pi]}(u_{i,j}(t),t)
\end{equation}
where $u_{i,j}:[0,2\pi]\to\C$ parametrises the $j$th strand of the $i$th component. The number $n_i$ denotes the number of strands forming the component $C_i$.
We say that $B$ is \textbf{P-fibered with multiplicities} $(m_1,m_2,\ldots,m_n)$ if the loop of polynomials
\begin{equation}
\tilde{g}(u,t):=\prod_{i=1}^n\prod_{j=1}^{n_i}(u-u_{i,j}(t))^{m_i}
\end{equation}
defines a fibration map $\arg(\tilde{g}):(\C\times S^1)\backslash B\to S^1$.
\end{definition}
Note that the order in the list of multiplicities $(m_1,m_2,\ldots,m_n)$ depends on the ordering of the components $C_i$. The definition above should be understood such that a braid is P-fibered with given multiplicities if there exists some ordering of the components $C_i$ such that the condition in the definition above is satisfied.

Thus a P-fibered braid is a P-fibered braid of multiplicities $(m,m,m,\ldots,m)$ for any $m$ and a P-fibered braid with $O$-multiplicity $m$ is a P-fibered braid of multiplicities $(1,1,\ldots,1,m)$, where one of the strands is the 0-strand, the only component with a multiplicity greater than 1. A braid is P-fibered with multiplicities $(m_1,m_2,\ldots,m_n)$ if and only if it is P-fibered with multiplicities $(km_1,km_2,\ldots,km_n)$ for any positive integer $k$. Links whose components come with assigned multiplicities have already been used in the context of holomorphic polynomials, where they are termed multilinks \cite{eisenbud}.

The satellite operation in \cite{bode:sat} is well-defined even if the braids $B_i$ do not have the same number of strands. It was only required in \cite{bode:sat} to prove the P-fiberedness of the resulting braid. We denote the resulting satellite braid as in \cite{bode:sat} by $\mathcal{B}(B;B_1,B_2,\ldots,B_n)$. The following statement is proved by exactly the same argument as Theorem 1.2 in \cite{bode:sat}.

\begin{theorem}\label{teo:sat}
Let $B$ be a P-fibered braid with multiplicities $(m_1,m_2,\ldots,m_n)$. Let $B_j$ , $j=1,2,\ldots,n$, be P-fibered braids  on $m_j$ strands with multiplicities $(m_{j,1},m_{j,2},\ldots,m_{j,n_j})$, where $n_j$ is the number of components of the closure of $B_j$. Then there are natural numbers $q_i$, $i=1,2,\ldots,n$, such that $\mathcal{B}(B_1^{r_1},B_2^{r_2},\ldots,B_n^{r_n})$ is a P-fibered braid with multiplicities 
\begin{equation}
(m_{1,1},m_{1,2},\ldots,m_{1,m_1},m_{2,1},m_{2,2},\ldots,m_{2,m_2},\ldots,m_{n,1},m_{n,2},\ldots,m_{n,m_n})
\end{equation}
for all natural numbers $r_i\geq q_i$. If $m_i=1$, we may take $q_i=1$.
\end{theorem}

Setting $m_i=s$ for all $i$ and $m_{j,i}=1$ for all $j$, $i$ reproduces Theorem 1.2 in \cite{bode:sat}. Theorem~\ref{teo:sat} above also allows us to perform successive satellite operations as in Theorem~\ref{thm:Pmult} and \cite{bode:sat} until we reach a P-fibered braid where all its multiplicities are 1. In this case, the closure of its square is a real algebraic link by \cite{bode:ralg, bode:sat}. (Of course, there is no need to stop with the satellite operations once all multiplicities are 1, since in this case all multiplicities can be taken to be $m>1$.) Theorem~\ref{thm:Pmult} also follows from repeatedly applying Theorem~\ref{teo:sat}. If $e$ denotes the trivial braid on one strand, then the braid $B(B_1,B_2,\ldots,B_N)$ is equal to 
\begin{equation}
\mathcal{B}(\mathcal{B}(\ldots\mathcal{B}(\mathcal{B}(B_N;e,e,\ldots,e,B_{N-1});e,e,\ldots,e,B_{N-2});\ldots;e,e,\ldots,e,B_2);e,e,\ldots,e,B_1).
\end{equation}

Even though both Theorem~\ref{thm:Pmult} and Theorem~\ref{teo:sat} prove that the closure of $B(B_1^{2r_1},B_2^{2r_2},\ldots,B_N^{2r_N})$ is a real algebraic link, the constructed polynomials are different. In Theorem~\ref{thm:Pmult} there are $N$ compact 1-faces in the Newton boundary, while Theorem~\ref{teo:sat} (in combination with the construction in \cite{bode:ralg}) yields radially weighted homogeneous polynomials, i.e., there is only one compact 1-face in the Newton boundary. 
Let $\mathbb{B}_s$ denote the braid group on $s$ strands and $\sigma_i$, $i=1,2,\ldots,s-1$, denote the Artin generators, representing positive half-twists between neighbouring strands.

We denote by $\imath_s:\mathbb{B}_s\to\mathbb{B}_{s+1}$ the inclusion homomorphism that sends an Artin generator $\sigma_i$, $i=1,2,\ldots,s-1$, of $\mathbb{B}_s$ to $\sigma_i$ in $\mathbb{B}_{s+1}$. Let $T$ be the braid on $s+1$ strands represented by the word $T=\sigma_{s}\sigma_{s-1}\ldots\sigma_2\sigma_{1}\sigma_1\sigma_2\ldots\sigma_{s}$.
\begin{proposition}\label{prop}
Let $B$ be any braid on $s$ strands. Then there is a positive integer $M$ such that the closure of $T^{2m}\imath_s(B^2)\in\mathbb{B}_{s+1}$ is real algebraic for all $m>M$.
\end{proposition}
\begin{proof}
Note that $T^{2m}\imath_s(B^2)$ is equal to $B(B_1^2, B_2^2)$ with $B_1=B$ and $B_2$ a braid on one strand, parametrised by $(\rme^{\rmi m t},t)\subset\mathbb{C}\times[0,2\pi]$. 

It follows from the proof of Lemma~\ref{lem:a} that there is a positive integer $M$ such that $B$ is P-fibered with $O$-multiplicity 0 and leading coefficient $a(t)=\rme^{\rmi m t}$ for all positive integers $m>M$. Therefore $B^2$ is P-fibered with $O$-multiplicity 0 and leading coefficient $a(2t)$. The braid $B_2^2$ is P-fibered with $O$-multiplicity $s$ and leading coefficient 1. This can be seen as follows. It is given by the roots of $g_2(u,\rme^{\rmi t})=u-a(2t)$. The only non-zero critical point of $u^sg_2(u,\rme^{\rmi t})$ for a fixed value $t$ is given by $c(t)=\tfrac{s}{s+1}\rme^{\rmi 2mt}$ and the corresponding critical value $v(t)=-\tfrac{s^s}{(s+1)^{s+1}}\rme^{\rmi 2m(s+1)t}$ satisfies $\tfrac{\partial\arg(v(t))}{\partial t}>0$ for all $t$.

Since the leading coefficient $a(2t)$ of $B^2$ is exactly the lowest order coefficient of $g_2$, the geometric braids $B^2$ and $B_2^2$ form a compatible sequence. Since both are $u$-even, we know that the closure of $B(B^2,B_2^2)=T^{2m}\imath_s(B^2)$ is real algebraic by Theorem~\ref{thm:main3}.



\end{proof}

\begin{obs}
Naturally, the same arguments imply the corresponding result for $M$ a negative integer and $m$ any integer smaller than $M$.
\end{obs}

\begin{proof}[Proof of Corollary~\ref{cor:sub}]
Every link is a sublink of the closure of some square of a braid, see for example \cite{bode:sat} or Figure~\ref{fig:sub}a). It is thus a sublink of the closure of $T^{2m}\imath_s(B^2)$ (see Figure~\ref{fig:sub}b) for $m=1$), which is a satellite link of the $(2,4m)$-torus link and is real algebraic if $m$ is sufficiently large as in Proposition~\ref{prop}.
\end{proof}

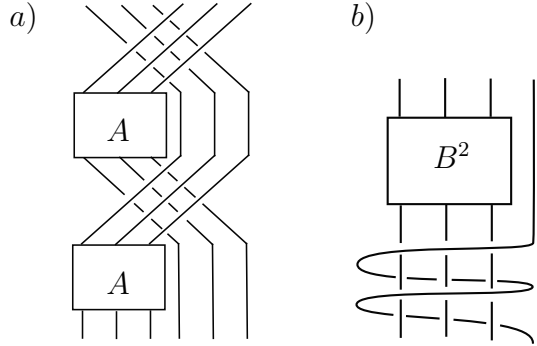
\begin{figure}[H]
\centering
\labellist
\Large
\pinlabel $a)$ at 20 1750
\pinlabel $b)$ at 1750 1750
\pinlabel $A$ at 500 400
\pinlabel $A$ at 500 1180
\pinlabel $B^2$ at 2200 1050
\endlabellist
\includegraphics[height=5cm]{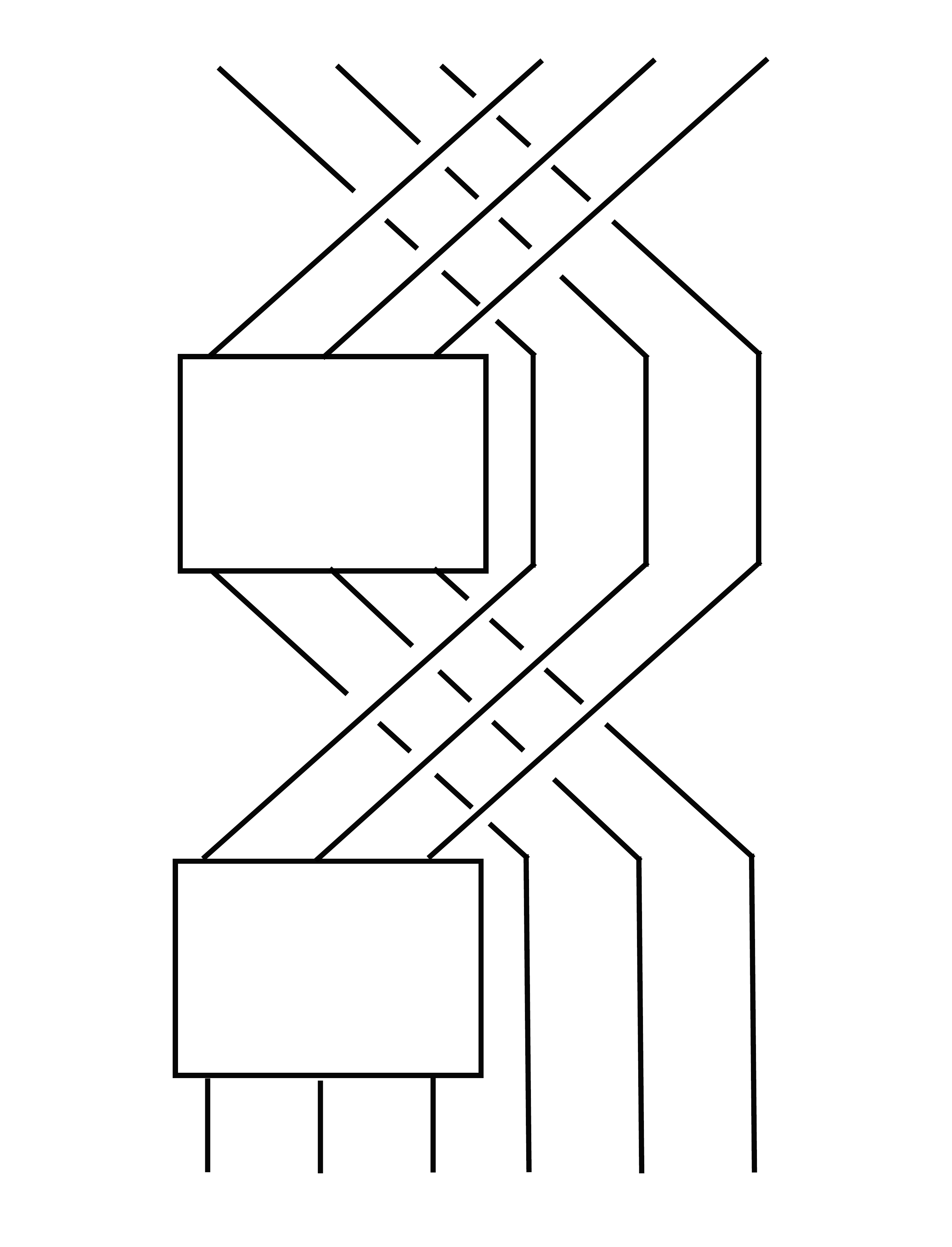}
\includegraphics[height=4cm]{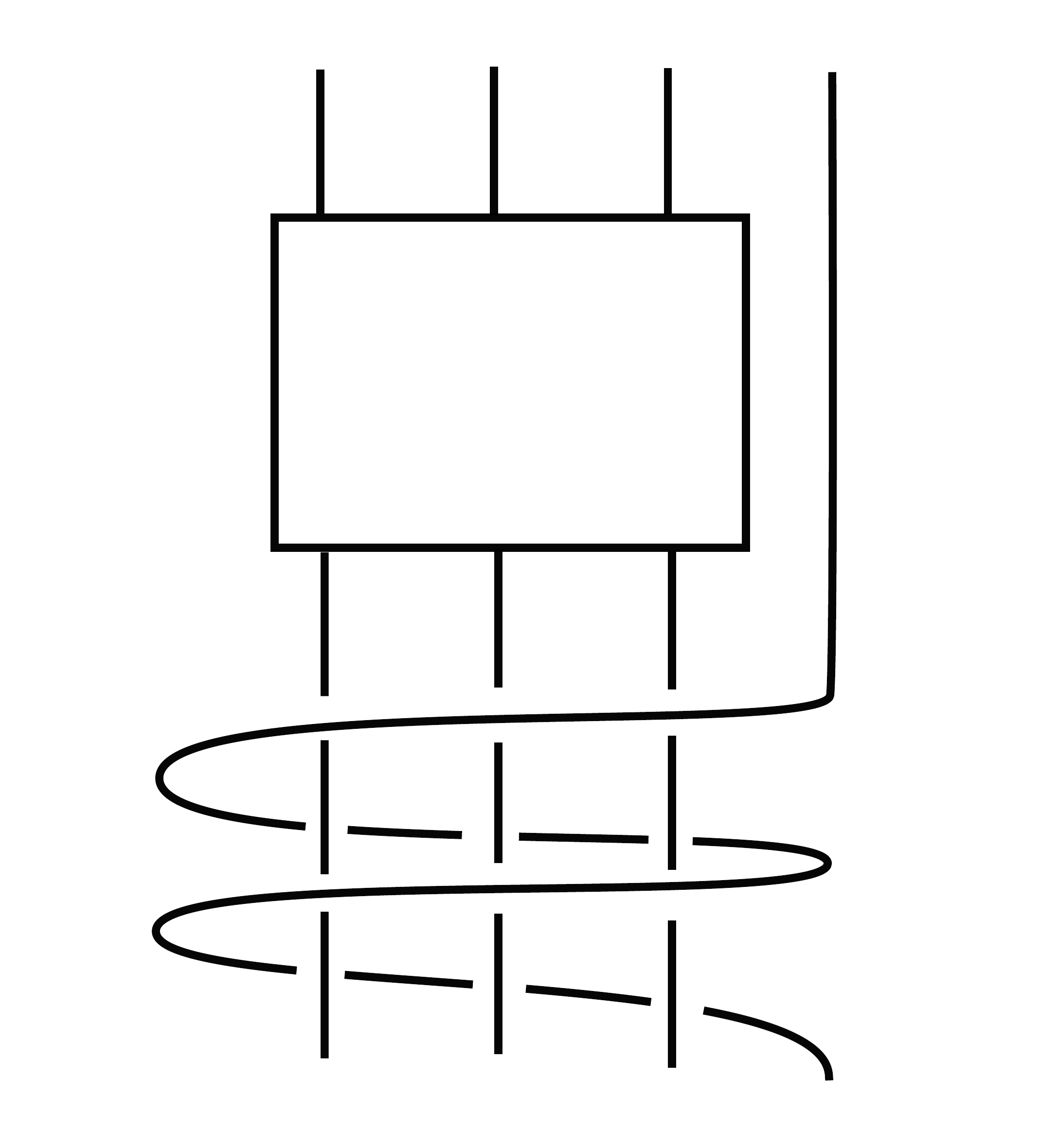}
\caption{a) A square braid $B^2$ whose closure consists of two copies of the closure of a braid $A$. b) $T^{2}\imath_s(B^2)$.\label{fig:sub}}
\end{figure}

The fact that every link is a sublink of a real algebraic link was also proved in \cite{bode:sat} and \cite{bode:thomo}. Corollary~\ref{cor:sub} shows that there is quite a large family of real algebraic links that contain a sublink, even within the family of satellites of torus links.

\section*{Acknowledgements}

The author is supported by the European Union's Horizon 2020 research and innovation programme through the Marie Sklodowska-Curie grant agreement 101023017. The author would like to thank Raimundo Nonato Araújo dos Santos, Osamu Saeki  and Eder Leandro Sanchez Quiceno for interesting and helpful discussions.

\end{document}